\documentclass[12pt]{article}

\usepackage{amssymb,amsmath,amsfonts}
\usepackage{graphicx,xcolor,enumitem}
\usepackage{epsfig}
\usepackage{amsthm}
\usepackage{bm}
\usepackage{subcaption}
\usepackage{csquotes}
\usepackage{multirow}
\renewcommand{\mkbegdispquote}[2]{\itshape}
\usepackage[letterpaper, twoside, hmarginratio=1:1, vmarginratio=1:1, left=1in,top=1in]{geometry}

\RequirePackage[breaklinks=true, hidelinks]{hyperref}
\usepackage{breakcites}

\usepackage{accents}
\newcommand{\ubar}[1]{\underaccent{\bar}{#1}}

\newcommand{\ang}[1]{\langle  #1 \rangle }
\newcommand{\cF}{{\mathcal F}}

\newcommand{\R}{\mathbb{R}}

\newcommand{\p}{\mathbb{P}}
\newcommand{\q}{\mathbb{Q}}
\newcommand{\E}{\mathbb{E}}
\newcommand{\F}{\mathbb{F}}
\newcommand{\bH}{\mathbb{H}}

\newcommand{\id}{{\mathbf 1}}
\newcommand{\as}{\mbox{{\rm a.s.}}}
\newcommand{\aee}{\mbox{{\rm a.e.}}}

\newtheorem{theorem}{Theorem}

\newtheorem{standassum}[theorem]{Standing Assumption}

\newtheorem{corollary}[theorem]{Corollary}

\newtheorem{definition}[theorem]{Definition}

\newtheorem{lemma}[theorem]{Lemma}

\newtheorem{proposition}[theorem]{Proposition}

\theoremstyle{definition}

\numberwithin{equation}{section}
\numberwithin{theorem}{section}

\begin{document}

\title{Robust Time-inconsistent Linear-Quadratic Stochastic Controls: A Stochastic Differential Game Approach}

\author{Bingyan Han\thanks{Thrust of Financial Technology, The Hong Kong University of Science and Technology (Guangzhou), China. Email: bingyanhan@hkust-gz.edu.cn}
	\and Chi Seng Pun\thanks{School of Physical and Mathematical Sciences, Nanyang Technological University, Singapore.  Email: cspun@ntu.edu.sg}
	\and Hoi Ying Wong\thanks{Department of Statistics, The Chinese University of Hong Kong, Hong Kong, China.  Email: hywong@cuhk.edu.hk}
}

\date{April 27, 2025}

\maketitle

\begin{abstract}
	This paper studies robust time-inconsistent (TIC) linear-quadratic stochastic control problems, formulated by stochastic differential games. By a spike variation approach, we derive sufficient conditions for achieving the Nash equilibrium, which corresponds to a time-consistent (TC) robust policy, under mild technical assumptions. To illustrate our framework, we consider two scenarios of robust mean-variance analysis, namely with state- and control-dependent ambiguity aversion. We find numerically that with time inconsistency haunting the dynamic optimal controls, the ambiguity aversion enhances the effective risk aversion faster than the linear, implying that the ambiguity in the TIC cases is more impactful than that under the TC counterparts, e.g., expected utility maximization problems.
	\\[2ex] 
	\noindent{\textbf {Keywords:} Robust stochastic controls, equilibrium controls, model uncertainty, time inconsistency, stochastic differential game, robust mean-variance analysis}
	\\[2ex]
	\noindent{\textbf {Mathematics Subject Classification:} 91A15, 49N90, 91A80, 91G10} \\
\end{abstract}

\section{Introduction}

Stochastic control theory embraces dynamic state evolution and decision-making problems arising in many disciplines. To capture the probabilistic uncertainty of states, the agent adopts stochastic processes under some specified probability measure for the modelling of states. Linear stochastic processes are perhaps the most parsimonious approximation of the reality, while they are widely adopted for their tractability; see \cite{yz99,fleming2006controlled,lz02,hjz12,hh17} for more details.

However, relying on a model to fully characterize the dynamic state evolution is unrealistic as the model is eventually estimated (or learned) and there will be unavoidable estimation and misspecification errors. This kind of uncertainty about model accuracy is known as model uncertainty or ambiguity. The seminal work of \cite{k21} clarified the subtle but essential difference between risk and ambiguity. Loosely speaking, risk and ambiguity are different layers of concepts; the former refers to the uncertainty about the future outcomes given their probabilistic properties (or distribution), while the latter refers to the uncertainty about the distribution. Typically, a rational decision-maker (DM) is averse to both risk and ambiguity.

To account for ambiguity in decision-making, an extensive list of research focuses on robust controls, which optimize against a set of alternative probabilistic models. Examples include problems with time-consistent (TC) objectives such as expected utility maximization \cite{m04,bp17} and with time-inconsistent (TIC) ones including mean-variance (MV) portfolio selection \cite{ih18}, where TC (resp. TIC)\footnote{In this paper, the abbreviation TC (resp. TIC) refers to time-consistent or time consistency (resp. time-inconsistent or time inconsistency) wherever it should be in place.} refers to the possession (resp. violation) of Bellman's principle of optimality for the dynamic control problem. The latter should be handled carefully since the concept of optimality is blurred. We refer the readers to the review of \cite{Cohen2020} that went through the developing literature on time preferences. For the aspect of solving the respective DM problems, \cite{Ekeland2006} appeared to be one of the first few works that establish a mathematical framework of subgame perfect equilibrium (SPE) on top of the ideas of \cite{Strotz1955,Phelps1968,Peleg1973}. The framework was then popularized since the seminal work of \cite{bc10} on its application to MV portfolio selection, while \cite{bm14,bkm17,hjz12,hjz17,Yong2012,Wei2017} subsequently developed the analogous dynamic programming and maximum principle frameworks for characterizing the SPE; see also the review of \cite{He2022} and \cite{bayraktar2021equilibrium,huang2021strong}. Motivated by the significance of model uncertainty, a developing literature introduces the robust SPE that accounts for ambiguity and time inconsistency simultaneously; see \cite{p18,Lei2020,hpw21FS,han2022robust}. Conceptually, both the DM and the fictitious adversary, as well as their future incarnations, form ``games in subgames" and the robust optimal policy for the DM (or ``DMs" across time) corresponds to her Nash equilibrium point.

Though, several questions are left open by the works of \cite{hpw21FS,han2022robust,Lei2021}. Technically, the preference ordering, advocated in aforementioned works that follow \cite{ahs03}, requires strong assumptions to validate the respective expansion results. The (locally) worst-case scenario must be determined before solving the equilibrium strategy. A framework without preference ordering is likely more tractable, while it still can provide insights into robust decision-making. The same motivation can be found in the studies of time-consistent objectives, such as \cite{ves03,os14}. Section \ref{sec:prefer} provides a detailed comparison between different formulations. Moreover, an interesting question is about the economic relation between ambiguity and risk aversion in the TIC cases. We could get some tastes from \cite{m04} that answered the TC counterpart question for expected utility maximization: the existence of ambiguity aversion increases the effective risk aversion linearly and partially explains the equity premium puzzle. In this work, we, however, discover different impacts of ambiguity aversion on the dynamic controls under TIC.

In this paper, we study a general stochastic linear-quadratic (LQ) control formulation. To account for the model misspecification, the DM introduces a set of alternative models with measures that are equivalent to the reference measure. Loosely speaking, the DM presumes that there is an adversary (nature or financial market) against him. Different from the previous results in \cite{hpw21FS,han2022robust}, we introduce a stochastic differential game (SDG) framework \textit{without preference ordering}. The DM and the adversary are in symmetric positions without any edge. In contrast, \cite{p18,hpw21FS,han2022robust} allowed the adversary to optimize based on the DM's actions. Thus, the DM in this paper is less pessimistic than the counterparts in the previous frameworks. By comparing our SDG framework with that of \cite{hpw21FS,han2022robust}, we conclude that ours is more flexible and tractable in the sense that it greatly relaxes the technical assumptions and facilitates new semi-analytic examples. Relevant literature on SDG is vast and we refer the readers to \cite{ves03,os14,moon2020linear} and references therein.

Mathematically speaking, the model uncertainty normally destroys the LQ structure of the linear state process and the quadratic objective. Especially, the objective implicitly depends on alternative measures. To overcome these difficulties, we incorporate the Radon-Nikod\'ym derivative as part of the state process. While the nonlinearity persists, we manage to apply the classical spike variation approach and obtain expansion results under mild assumptions that are comparable to the non-robust cases in \cite{yz99,hjz12}. Moreover, the sufficient conditions for an equilibrium of two players, the DM and the imaginary adversary, are in a parallel position as there is no preference ordering.  

To illustrate our LQ framework, we present two new examples with semi-analytic solutions nesting the MV portfolio selection as a special case. In the same spirit of considering state-dependent risk aversion as in \cite{bmz14}, the first example introduces a state-dependent ambiguity aversion. The second example, however, presents a new setting of control-dependent ambiguity aversion. Both settings are sensible in the dimensional analysis on objective. However, they demonstrate different behaviors in the numerical analysis. An essential technical difficulty is that the derived ODE systems are more complicated than the Riccati-type ODEs studied in the non-robust setting \cite{yz99,hjz12}. We present the existence and uniqueness results with a truncation method under some technical assumptions. 

One main economic finding of this paper is the distinct impacts of model uncertainty on the equilibrium, which validates that the DM with equilibrium controls should seriously address the model uncertainty. In contrast, the DM with expected utility maximization can simply increases the risk aversion coefficient (linearly) to reflect the ambiguity aversion, as shown in \cite{m04}. It provides a new insight into the equity premium puzzle that argues for the requirement of unrealistically high level of risk aversion. Remarkably, our numerical analysis demonstrates a rate faster than the linear with state-dependent ambiguity aversion in increasing the effective risk aversion. The impact of ambiguity aversion is remarkable so that an ambiguity-averse investor even with a low risk aversion could invest a similar proportion into stocks as with an ambiguity-neutral investor with a high risk aversion. Practitioners usually invest less aggressively in stocks than the classical portfolio theory suggests. While this phenomenon was explained in \cite{m04} with either an excessive ambiguity aversion or an excessive risk aversion, our work, however, finds that rational investors can invest conservatively even under both fair ambiguity and risk aversions. Nevertheless, DMs may find the consideration of state-dependent ambiguity aversion too conservative since it leads to a zero-investment strategy in the no-risk-aversion setting. In this case, our new control-dependent ambiguity aversion may agree with their tastes of ambiguity, as this setting yields a closer behavior to the classical MV portfolios.

The remainder of this paper is organized as follows. Section \ref{sec:form} formulates the problem in a SDG framework. Section \ref{sec:suff} proves sufficient conditions for an equilibrium. Section \ref{sec:examples} applies the derived framework to examples with state-dependent and control-dependent ambiguity aversions, respectively. Section \ref{sec:num} numerically shows the complex interplay between ambiguity and risk aversion. Our code is available at \url{https://github.com/hanbingyan/DiffGame}. Section \ref{sec:con} concludes. All proofs are deferred to the Appendix \ref{sec:app}.

\section{Problem Formulation}\label{sec:form}
Let $T>0$ be a fixed finite terminal time. Consider a probability space $(\Omega, \cF, \F, \p)$, where $\F := (\cF_t)_{t\in[0,T]}$ is the completed canonical filtration of  a $d$-dimensional standard $\p$-Brownian motion $W^\p_t=(W^\p_{1t},\cdots,W^\p_{dt})'$, $t\in[0,T]$. We summarize the frequently used notations below.  Consider a set of probability measures that are equivalent to $\p$, denoted by $\mathcal{Q} := \{\q | \q \sim \p\}$. In fact, we work on a set smaller than $\mathcal{Q}$ later. Depending on the context, $\bH$ represents the Euclidean spaces or the spaces of continuous functions. Let $k,p$ be generic positive integers, we denote by
\begin{itemize}[label=\raisebox{0.25ex}{\tiny$\bullet$},leftmargin=1.5em]
	\item $\odot$ element-wise multiplication, 
	\item $\ang{\cdot, \cdot}$ the inner product, 
	\item $|A| =\sqrt{\sum_{i,j}a_{ij}^2}$ the Frobenius norm of a matrix $A$,
	\item $\mathbb S^k$ the set of all  $k \times k$ real symmetric matrices,
	\item $\E^\p[\cdot]$ and $\E^\p_t[\cdot]$ the expectation and conditional expectation under $\p$,
	\item $L^\infty(t, T; \bH) := \left\{ f: [t, T] \rightarrow \bH \,\big| \text{ $f$ is measurable and essentially bounded} \right\}$, 
	\item $L^p(t, T; \bH) := \left\{ f: [t, T] \rightarrow \bH \,\big| \int^T_t |f_s|^p ds < \infty \right\}$,
	
	\item $L^\infty_{\cF_t}(\Omega; \bH) := \left\{ X: \Omega \rightarrow \bH \,\big| \text{ $X$ is $\cF_t$-measurable and essentially bounded} \right\}$, 
	\item $L^p_{\cF_t}(\Omega; \bH) := \left\{ X: \Omega \rightarrow \bH \,\big| \text{ $X$ is $\cF_t$-measurable and } \sup_{\q \in \mathcal{Q}} \E^\q[|X|^p] < \infty \right\}$,
	
	\item $L^\infty_\F(t, \, T; \, \bH) := \big\{ X: [t, T] \times \Omega \rightarrow \bH \,\big| (X_s)_{s\in[t,T]} \text{ is } \F\text{-progressively measurable and} \\ \text{ essentially  bounded}\big\}$,
	\item $L^p_\F(t, \, T; \, \bH) := \big\{ X: [t, T] \times \Omega \rightarrow \bH \,\big| (X_s)_{s\in[t,T]} \text{ is } \F\text{-progressively measurable and } \\ \sup_{\q \in \mathcal{Q}} \E^\q[\int_t^T |X_s|^p ds ]<\infty  \big\}$,
	
	\item $L^p_\F(\Omega; \, C(t, \, T; \, \bH)) := \big\{ X: [t, T] \times \Omega \rightarrow \bH \,\big| (X_s)_{s\in[t,T]} \text{ is continuous, }$ $\F$-adapted, and $ \sup_{\q \in \mathcal{Q}} \E^\q[\sup_{s\in [t,T]} |X_s |^p]<\infty \big\}$.
\end{itemize}

\subsection{The Models and the Controls}
Under the reference measure $\p$, consider an $n$-dimensional linear controlled state process $X$ satisfying the stochastic differential equation (SDE):
\begin{equation}\label{eq:def-state}
	dX_s  =  (A_sX_s+B_s'u_s+b_s)ds +\sum_{j=1}^d(C_s^jX_s+D_s^{j} u_s+\sigma_s^{j} )dW^\p_{js}, \quad X_0  =  x_0,
\end{equation}
where $x_0\in \R^n$ is the initial value  and $(u_t)_{t\in[0,T]}$ is the $l$-dimensional control process. We assume parameters are deterministic functions of suitable sizes, satisfying conditions detailed in Assumption \ref{assum:state} below. For problems with random coefficients, one may augment the state process with dynamics of the respective stochastic processes. Then our framework is still applicable if the augmented processes have a linear structure.

Besides the risk uncertainty captured by the randomness of the state process, DMs also confront uncertainty from potential model misspecification and ambiguity about the correctness of \eqref{eq:def-state}. \cite{k21} and the paradox by \cite{e61} emphasize the importance of ambiguity aversion in decision-making. A widely used robust control rule is initialized and formalized by \cite{w45} and \cite{ahs03}. Instead of considering a single reference measure $\p$, an ambiguity-averse DM considers a set of alternative models and optimizes against the worst-case scenario among the models. Specifically, for each $\q\in \mathcal{Q}$, there exists an $\F$-progressively measurable stochastic process $h_s = (h_{1s},...,h_{ds})'$ such that for $t\in[0,T]$, the Radon-Nikod\'ym derivative of $\q$ with respect to $\p$ is given by
\begin{equation}\label{eq:radon}
	\zeta_t := \frac{d\q}{d\p} \Big\vert_{\cF_t} = \exp\left(\int^t_0 h'_s dW^\p_s -\frac{1}{2} \int^t_0 h'_s h_s ds\right).
\end{equation}
For a later use, we introduce $\zeta(s;t) := \zeta_s/\zeta_t$. Technically, we restrict that $h \in L^\infty_\F(0,T;\R^d)$ and thus the Novikov's condition is satisfied. Therefore, the Radon-Nikod\'ym derivative is a positive $\p$-martingale and $W^\q_t$, driven by $dW^\q_t = dW^\p_t - h_t dt$, is a $d$-dimensional $\q$-Brownian motion defined on $(\Omega,\cF, \F,\q)$. For any fixed $\q\in\mathcal{Q}$, the state process starting at time $t\in [0,T)$ with the current state $x_t \in L^4_{\cF_t}(\Omega; \, \R^n)$ is 
\begin{align} 
	& dX_s = \mu(X_s, u_s, h_s, s) ds + \sum_{j=1}^d (C_s^jX_s+D_s^{j} u_s+\sigma_s^{j} )dW^\q_{js}, \quad X_t = x_t, \label{eq:state_Q} \\
	\text{with} & \quad \mu(X_s, u_s, h_s, s) :=  (A_sX_s+B_s'u_s+b_s) +\sum^d_{j=1}(C_s^jX_s+D_s^{j} u_s+\sigma_s^{j} )h_{js}. \label{eq:mu}
\end{align}
The $(h_s)_{s\in[t,T]}$ in \eqref{eq:state_Q} reflects the distortions on the reference model and captures a set of alternative ones. Note that under each $\q$, we lose the linear structure of the state process \eqref{eq:state_Q} when we regard $h$ as a control by the imaginary adversary, noting the $u_sh_{js}$.

\begin{standassum}\label{assum:state}
	Suppose the parameters in \eqref{eq:def-state} satisfy $A, C^j \in L^\infty(0, T; \R^{n\times n})$, $B \in L^\infty(0, T; \R^{ l\times n})$,  $D^j \in L^\infty(0, T; $ $ \R^{n\times l})$, and $b, \sigma^j $ $\in L^4(0, T; \R^n)$, $j = 1, ..., d$.
\end{standassum}

\begin{definition}\label{def:admi}
	A pair of $(h, u)$ is called admissible if 
	\begin{itemize}
		\item[(1)] the control process $u \in L^4_\F(0, \, T; \, \R^l)$;
		\item[(2)] $h \in L^\infty_\F(0,T;\R^d)$.
	\end{itemize}
\end{definition}

Under Assumption \ref{assum:state}, there exists a unique solution $X \in L^4_\F(\Omega; \, C(0, \, T; \, \R^n))$ to \eqref{eq:def-state} for admissible $(h, u)$; see \cite[Appendix D]{fleming2006controlled}. Compared with the standard setting in \cite[Chapter 6]{yz99}, we impose stronger moment conditions in Standing Assumption \ref{assum:state} and Definition \ref{def:admi} to validate the well-posedness of adjoint processes. 

The DM considers the following objective function which relies on $\q$:
\begin{align}
	J(t, x_t; u, h) := &  ~\mathbb{E}^\q_t \left[ \int_t^T f(X_s, u_s, h_s, s) ds \right] +\frac{1}{2}\mathbb{E}^\q_t [\ang{G X_T, X_T}]\nonumber \\
	& -\frac{1}{2} \ang{\nu \mathbb{E}^\q_t [X_T], \mathbb{E}^\q_t [X_T]}-\ang{ \mu_1 x_t+\mu_2, \mathbb{E}^\q_t [X_T]}, \label{eq:obj}
\end{align}
where $X=X^{t,x_t,u,h}$ is solved from \eqref{eq:state_Q}, $\mathbb{E}^\q_t [\cdot]=\mathbb{E}^\q [\cdot|\cF_t]$ is the conditional expectation under probability measure $\q$, and
\begin{align*}
	f(X_s, u_s, h_s, s) & := \frac{1}{2} \Big(\ang{Q_sX_s, X_s}+\ang{ R_su_s, u_s} - h'_s \Phi (X_s, u_s, s) h_s \Big), \\
	\Phi (X_s, u_s, s) & :=\text{diag} (\Phi_1 (X_s, u_s, s), ..., \Phi_d (X_s, u_s, s)),\\
	\Phi_j (X_s, u_s, s) & := \ang{\Xi_{1, j}X_s, X_s} + \ang{ \Xi_{2, j}u_s, u_s} + \xi_{3, j}.
\end{align*}
$\Phi(\cdot)$ represents the ambiguity term. Recall that $x_t \in L^4_{\cF_t}(\Omega; \, \R^n)$ denotes the current state and can be random in \eqref{eq:state_Q}.

Similarly, parameters in the objective are also required to be deterministic and to satisfy the Assumption \ref{assum:obj} below throughout the paper. For a symmetric matrix $Q$, we write $Q \succeq 0$ if $Q$ is positive semidefinite.
\begin{standassum}\label{assum:obj}
	Suppose 
	\begin{itemize}
		\item[(1)] $Q \in L^\infty(0, \, T; \, {\mathbb S}^n)$ and $R \in L^\infty(0, \, T; \, {\mathbb S}^l)$, while $\Xi_{1, j}, G, \nu \in \mathbb S^n$, $\Xi_{2, j} \in {\mathbb S}^l$, $\mu_1 \in \R^{n\times n}$, $\mu_2 \in \R^n$, and $\xi_{3, j} \in \R $ are constants. 
		\item[(2)] $Q\succeq 0, \aee$ and $R\succeq 0, \aee$, while $G\succeq 0$, $\nu \succeq 0$, $\Xi_{1, j} \succeq 0$, $\Xi_{2, j} \succeq 0$, and $\xi_{3, j} \geq 0$.	
	\end{itemize}
\end{standassum}

There are two essential challenges originated from $h$. The first one is the non-linearity in the state \eqref{eq:state_Q} as discussed before. The second one is that the objective also relies on $h$ via the alternative measure $\q$. Our key idea is to augment the state process $X$ by the Radon-Nikod\'ym derivative $\zeta(s;t) = \zeta_s/\zeta_t$ in \eqref{eq:radon}. Specifically, we consider the coupled controlled state processes: for $s\in[t,T]$,
\begin{equation}\label{eq:aug-state}
	\left\{
	\begin{array}{rcll}
		dX_s  &=&  (A_sX_s+B_s'u_s+b_s)ds \\
		& &  +\sum_{j=1}^d(C_s^jX_s+D_s^{j} u_s+\sigma_s^{j} )dW^\p_{js}, & X_t = x_t, \\
		d\zeta(s;t) &=& \zeta(s;t) h'_s dW^\p_s, & \zeta(t;t) = 1.
	\end{array}\right.
\end{equation}

Then, we can express the objective \eqref{eq:obj} as an augmented cost functional
\begin{align}
	J(t, x_t, \zeta; u, h) & = \E^\p_t \left[\int_t^T\zeta(s;t) f(X_s, u_s, h_s, s) ds \right] +\frac{1}{2} \E^\p_t [\zeta(T;t) \ang{G X_T, X_T}] \label{eq:aug-obj} \\
	&\quad -  \frac{1}{2} \ang{\nu \E^\p_t [\zeta(T;t)X_T], \E^\p_t [\zeta(T;t) X_T]}-\ang{ \mu_1 x_t+\mu_2, \E^\p_t [\zeta(T;t) X_T]}. \nonumber
\end{align}
Clearly, $J(t, x_t, \zeta; u, h) = J(t, x_t; u, h)$. Here, the notation $\zeta$ highlights the dependence on $\zeta(\cdot; t)$ and should not be interpreted as the initial value of $\zeta(\cdot; t)$.

\eqref{eq:aug-obj} shows the dependence on $h$ explicitly, while the objective is beyond quadratic in the states and the controls. This expression unifies the measure used in the objective but it will cause technical difficulties in proving expansion results in Theorem \ref{Thm:hvariate} and \ref{Thm:uvariate}. \eqref{eq:aug-obj} demonstrates that preferences of the DM change in a temporally inconsistent fashion as the time involves. The tower rule fails due to the quadratic term of the conditional expectation. State dependence on $x_t$ is another source of TIC.

Furthermore, the problem concerning $h$ is time-inconsistent. We demonstrate this using the MV objective, a special case of \eqref{eq:aug-obj}. Given a control $u$, consider
	\begin{equation}
		\begin{aligned}
			& \sup_h \Big( \frac{1}{2} \E^\p_t \left[\zeta(T;t) X^2_T \right] -  \frac{1}{2} (\E^\p_t [\zeta(T;t)X_T])^2 - \mu_2 \E^\p_t [\zeta(T;t) X_T] \\
			& \qquad - \frac{\xi}{2} \E^\p_t \Big[\int_t^T\zeta(s;t)  h'_s h_s ds \Big] \Big) \label{eq:state_onzeta} \\
			& = \frac{1}{\zeta_t} \sup_h  \Big( \frac{1}{2} \E^\p_t \left[\zeta_T X^2_T \right] -  \frac{1}{2 \zeta_t} (\E^\p_t [\zeta_T X_T])^2 - \mu_2 \E^\p_t [\zeta_T X_T] - \frac{\xi}{2} \E^\p_t \Big[\int_t^T\zeta_s  h'_s h_s ds \Big] \Big).
		\end{aligned}
	\end{equation} 
After factoring out $1/\zeta_t$, \eqref{eq:state_onzeta} still exhibits state dependence on $\zeta_t$ due to the second term from the variance operator. Hence, the problem in \eqref{eq:state_onzeta} is time-inconsistent, which justifies the equilibrium definition for $h$ in \eqref{ineq:h*}.

This paper takes TIC seriously and introduces a game-theoretic framework in line with \cite{hjz12,bkm17,p18,hpw21FS}. Given an admissible control pair $(h^*, u^*)$, for any $t\in [0,T)$ and $\varepsilon>0$, we perturb the admissible control pair with $v\in L^4_{\cF_t}(\Omega; \, \R^l)$ and $\eta \in L^\infty_{\cF_t}(\Omega; \, \R^d)$, defined as the $\varepsilon$-policy pair by
\begin{eqnarray}
	h^{t,\varepsilon,\eta}_s & = & h^*_s+\eta \id_{s\in [t,t+\varepsilon)},\;\;\;s\in[t,T],  \label{hspike} \\
	u^{t,\varepsilon,v}_s & = & u^*_s+v\id_{s\in [t,t+\varepsilon)},\;\;\;s\in[t,T]. \label{uspike}
\end{eqnarray}
If $(h^*, u^*)$ is an equilibrium, it should be a saddle point in the following local sense:
\begin{definition}\label{def:equili}
	Let $h^* \in L^\infty_\F(0,T;\R^d)$ and $u^* \in L^4_\F(0, \, T; \, \R^l)$ be given controls. Let $(X^*, \zeta^*)$ be the corresponding state processes in \eqref{eq:aug-state}. The control pair $(h^*, u^*)$ is called an equilibrium if for $t\in [0,T)$,
	\begin{align} 
		& \limsup_{\varepsilon\downarrow 0} \frac{J(t, X^*_t, \zeta^*; u^*, h^{t,\varepsilon,\eta})-J(t, X^*_t, \zeta^*; u^*, h^*)}{\varepsilon}\le 0, \quad \p\text{-}\as, \label{ineq:h*} \\
		& \liminf_{\varepsilon\downarrow 0} \frac{J(t, X^*_t, \zeta^*; u^{t,\varepsilon,v},h^*)-J(t, X^*_t, \zeta^*; u^*,h^*)}{\varepsilon}\ge 0, \quad \p\text{-}\as, \label{ineq:u*}
	\end{align}
	where $(h^{t,\varepsilon,\eta}, u^{t,\varepsilon,v})$ are defined in \eqref{hspike}-\eqref{uspike} for any perturbations $\eta \in L^\infty_{\cF_t}(\Omega; \, \R^d)$ and $v\in L^4_{\cF_t}(\Omega; \, \R^l)$.
\end{definition}

We compare Definition \ref{def:equili} with several alternative formulations:
	\begin{itemize}
		\item One alternative definition is given by 
		\begin{align}\label{alterdef} 
			& \liminf_{\varepsilon\downarrow 0} \frac{\sup_h J(t, X^*_t, \zeta; u^{t, \varepsilon, v}, h) - \sup_h J(t, X^*_t, \zeta; u^*, h)}{\varepsilon} \geq 0, \quad \p\text{-}\as
		\end{align}
		This approach ignores time inconsistency on $h$ and optimizes over $h$ at each time $t$. The optimizer for $\sup_h J(t, X^*_t, \zeta; u^*, h)$, if it exists, is not necessarily the worst-case scenario for the problem at time $t + \varepsilon$. Due to the time inconsistency highlighted in \eqref{eq:state_onzeta}, the agent will always have self-contradictory views on the worst-case scenario under this definition. To account for time inconsistency on $h$, an equilibrium $h^*$ should be defined in a suitable manner, and $J(t, X^*_t, \zeta^*; u^*, h^*)$ should be used instead of $\sup_h J(t, X^*_t, \zeta; u^*, h)$.
		
		\item \cite{li2023robust} introduces a robust equilibrium strategy in a closed-loop (feedback) setting. Their definition closely resembles that in \cite{hpw21FS}.  In particular, the worst-case scenario in \cite[Definition 1]{li2023robust} is also defined as an equilibrium, addressing time inconsistency on $h$ similarly to our condition \eqref{ineq:h*}. \cite{li2023robust} imposes the preference ordering in which $h^*$ is solved before $u^*$. Additionally, \cite{li2023robust} assumes that drift and volatility terms are within a given set. In contrast, \cite{hpw21FS} employs an entropy penalty.
	\end{itemize}

\subsection{Comparison between Different Equilibrium Controls}\label{sec:prefer}
Before presenting the sufficient conditions for an equilibrium, we compare our Definition \ref{def:equili} on the equilibrium control pair with similar considerations in \cite{hpw21FS,han2022robust}. Intuitively, preference ordering means whether the DM or the fictitious adversary makes choices first.

\begin{itemize}[label=\raisebox{0.25ex}{\tiny$\bullet$}, leftmargin=1.5em]
	\item Definition \ref{def:equili} gives an \textbf{open-loop equilibrium without preference ordering}. For certain cases, an open-loop equilibrium may have a closed-loop representation \cite{sun2016,wang2020}. To highlight the distinction of our formulation, Table \ref{tab:open} gives closed-loop representations of \eqref{hspike}-\eqref{uspike}, if they exist. Here, $X^*$ is a strong solution corresponding to $(h^*, u^*)$. Crucially, perturbations on open-loop controls do not alter the state process inside $(h^*, u^*)$. Besides, the order of solving $h^*$ and $u^*$ does not affect the equilibrium control in Definition \ref{def:equili}, which can be interpreted as a SDG formulation. 
	\begin{table}[!ht]
		\centering
		\begin{tabular}{c|c}
			\hline
			Perturbation on $h^*$ & Perturbation on $u^*$ \\
			\hline
			$u^*_s=u^*(X^*_s, s)$ & $u^{t,\varepsilon,v}_s=u^*(X^*_s, s) + v\id_{s\in [t,t+\varepsilon)}$ \\
			$h^{t,\varepsilon,\eta}_s=h^*(X^*_s, s) + \eta \id_{s\in [t,t+\varepsilon)}$ & $h^*_s=h^*(X^*_s, s)$ \\
			\hline
		\end{tabular}
		\caption{Open-loop controls without preference ordering}\label{tab:open}
	\end{table}
	
	\item \cite{han2022robust} considers an \textbf{open-loop equilibrium with preference ordering}. In this case, the maximization over $h^*$ is solved before the minimization over $u^*$. This suggests that when the control $u$ changes, the worst-case scenario can be adjusted in advance, which is also used in \cite[Definition 2]{li2023robust}. With slightly abuse of notations, denote by $X^*$ the state process corresponding to $(h^*, u^*)$ and by $\tilde X^*$ the state process with the $u^{t, \varepsilon, v}$ and $h^*(\cdot, u^{t,\varepsilon,v}_s, s)$ in Table \ref{tab:preforder}. Though, this formulation is still an open-loop control since $\eta$ does not affect the state process in $h^*$ and $v$ does not affect the state process in $u^*$. Note that the representation $u^*(X^*_s, s)$ and $h^*(\tilde X^*_s, u^{t,\varepsilon,v}_s, s)$ are obtained by the sufficient conditions in \cite{han2022robust}, instead of given a priori.
	\begin{table}[!ht]
		\centering
		\begin{tabular}{c|c}
			\hline
			Perturbation on $h^*$ & Perturbation on $u^*$ \\
			\hline
			$u^{t,\varepsilon,v}_s =  u^*(X^*_s, s) + v\id_{s\in [t,t+\varepsilon)}$ & $u^{t,\varepsilon,v}_s = u^*(X^*_s, s) + v\id_{s\in [t,t+\varepsilon)}$ \\
			$h^{t,\varepsilon,\eta}_s = h^*(\tilde X^*_s, u^{t,\varepsilon,v}_s, s) + \eta \id_{s\in [t,t+\varepsilon)}$ & $h^*_s = h^*( \tilde X^*_s, u^{t,\varepsilon,v}_s, s)$ \\
			\hline
		\end{tabular}
		\caption{Open-loop controls with preference ordering}
		\label{tab:preforder}
	\end{table}
	\item \cite{bmz14,huang2017char,hpw21FS} consider \textbf{a closed-loop (or feedback) equilibrium with preference ordering}. Perturbations under a closed-loop control formulation alter the state process inside the controls. Denote by $X^{t, \varepsilon, v}$ the state process with $(u^{t, \varepsilon, v}(\cdot, s), h^*(\cdot, u^{t,\varepsilon,v}_s, s))$ and by $X^{t, \varepsilon, v, \eta}$ the state process with $u^*(\cdot, s) + v\id_{s\in [t,t+\varepsilon)}$ and $h^*(\cdot, u^{t,\varepsilon,v}_s, s) + \eta\id_{s\in [t,t+\varepsilon)}$. $u^*$ is usually specified as a linear function of $X^*$ in the LQ setting. Table \ref{tab:closed} shows that perturbations $\eta$ and $v$ change the controls substantially, compared with open-loop counterparts. 
	\begin{table}[!ht] 
		\centering
		\begin{tabular}{c|c}
			\hline 
			Perturbation on $h^*$ & Perturbation on $u^*$ \\
			\hline
			$u^{t,\varepsilon,v}_s = u^*(X^{t, \varepsilon, v, \eta}_s, s) + v\id_{s\in [t,t+\varepsilon)}$ & $u^{t,\varepsilon,v}_s = u^*(X^{t,\varepsilon, v}_s, s) + v\id_{s\in [t,t+\varepsilon)}$ \\
			$h^{t,\varepsilon,\eta}_s = h^*(X^{t, \varepsilon, v, \eta}_s, u^{t,\varepsilon,v}_s, s) + \eta \id_{s\in [t,t+\varepsilon)}$ & $h^*_s = h^*(X^{t,\varepsilon, v}_s, u^{t,\varepsilon,v}_s, s)$ \\
			\hline
		\end{tabular}
		\caption{Closed-loop controls with preference ordering}
		\label{tab:closed}
	\end{table}
\end{itemize}

From an economic perspective, the framework in this paper treats the adversary controlling $h$ and the DM selecting $u$ in symmetric positions. The DM without preference ordering is less pessimistic compared to the DM with the preference ordering that allows the adversary to act based on the DM's choice. Hence, our framework leads to more optimistic decisions. Technically, the framework is more flexible and expansion results hold under less restrictive assumptions than \cite{hpw21FS,han2022robust}.

\section{Sufficient Conditions}\label{sec:suff}
To derive sufficient conditions for an equilibrium pair satisfying Definition \ref{def:equili}, we adopt the stochastic maximum principle approach \cite{p90,yz99,hjz12,hpw21FS}. Without the preference ordering, the sufficient conditions in Theorems \ref{Thm:hvariate} and \ref{Thm:uvariate} can be obtained simultaneously, without the need to solve for $h^*$ or $u^*$ beforehand. This edge greatly simplifies our technical assumptions and embraces more examples with semi-analytic solutions.

To characterize $h^*$, define the first-order adjoint process $(p^\zeta(\cdot;t), k^\zeta(\cdot;t))$ satisfying the backward stochastic differential equation (BSDE):
\begin{align}
	\left\{\begin{array}{lcl}
		dp^\zeta(s;t) & = & -\big( (h^*_s)' k^\zeta(s;t) + f(X^*_s, u^*_s, h^*_s, s) \big) ds + k^\zeta(s;t)'dW^\p_s, \quad s\in[t,T],\\
		p^\zeta(T;t) & = & \frac{1}{2} \ang{G X^*_T, X^*_T} - \ang{\nu \mathbb{E}^\p_t [\zeta^*(T;t) X^*_T], X^*_T} - \ang{ \mu_1 x_t + \mu_2, X^*_T}. \label{pzeta}
	\end{array}\right.
\end{align}
Since $u^* \in L^4_\F(0, \, T; \, \R^l)$ and $X^* \in L^4_\F(\Omega; \, C(0, \, T; \, \R^n))$, then we have $f(X^*, u^*, h^*, \cdot) \in L^2_\F(0,T; \R)$ and the boundary condition $p^\zeta(T;t) \in L^2_{\cF_T}(\Omega; \R)$. By \cite[Theorem 2.2, Chapter 7]{yz99}, the BSDE \eqref{pzeta} admits a unique solution $(p^\zeta(\cdot;t), k^\zeta(\cdot;t)) \in L^2_\F(t,T; \R) \times L^2_\F(t,T; \R^d)$.

The following Theorem \ref{Thm:hvariate} expands the cost functional with respect to the perturbation of $h^*$ and gives a sufficient condition of $h^*$.
\begin{theorem}\label{Thm:hvariate}
	For $t\in [0,T)$ and $s \in [t, T]$, define
	\begin{equation}
		\Lambda^\zeta(s;t) := \zeta^*(s;t) k^\zeta(s;t) - \zeta^*(s;t)\Phi(X^*_s, u^*_s, s)h^*_s.
	\end{equation}
	If $h^*$ validates
	\begin{equation}\label{eq:h-suff}
		\Lambda^\zeta( t; t) = 0, \quad \as, \aee, t \in [0, T],
	\end{equation}
	then inequality \eqref{ineq:h*} holds for $h^*$.
\end{theorem}
Compared to Theorem \ref{Thm:uvariate}, the second-order condition in Theorem \ref{Thm:hvariate} holds automatically due to the variance operator and the assumption that $\Phi \succeq 0$. As a result, there is no need to introduce the corresponding second-order adjoint process.

To prove a sufficient condition of $u^*$, introduce $(p^x(\cdot;t), (k^{x, j}(\cdot;t))_{j=1,\cdots, d})\in L^2_\F(t,T;\R^n) \times (L^2_\F(t,T;\R^n))^d$ as the solution to
\begin{align}
	\left\{\begin{array}{ll}
		dp^x(s;t) = & -\big( A'_s p^x(s;t) + \sum^d_{j=1}(C_s^j)' k^{x,j}(s;t) + \zeta^*(s;t) f^*_x(X^*_s, u^*_s, h^*_s, s) \big) ds \\
		& +\sum_{j=1}^d k^{x,j}(s;t)dW^\p_{js},  \quad s\in[t,T],\\
		p^x(T;t) = & \zeta^*(T;t) G X^*_T - \zeta^*(T;t) \nu \mathbb{E}^\p_t [\zeta^*(T;t) X^*_T] - \zeta^*(T;t) (\mu_1 x_t+\mu_2). \label{eq:px}
	\end{array}\right.
\end{align}
Define the second-order adjoint process $(P(\cdot;t), (K^j(\cdot;t))_{j=1,\cdots,d})\in L^2_\F(t,T;\mathbb S^n)$ $\times$ $ (L^2_\F(t,T;\mathbb S^n))^d$ as the solution to BSDE \eqref{eq:P}:
\begin{align}
	\left\{\begin{array}{lll}
		dP(s;t) &=& -\Big( A'_s P(s;t) + P(s;t) A_s + \zeta^*(s;t) Q_s \\
		& & \hspace{0.5cm} - \frac{1}{2} \zeta^*(s;t) (h^*_s)' \Phi_{xx} (X^*_s, u^*_s, s) h^*_s \\
		& & \hspace{0.5cm} +\sum_{j=1}^d \big( (C^j_s)'P(s;t)C^j_s + (C^j_s)'K^j(s;t) + K^j(s;t) C^j_s \big) \Big) ds\\
		& & +\sum_{j=1}^dK^j(s;t)dW^\p_{js}, \quad s\in[t,T],\\
		P(T;t) & = & \zeta^*(T;t) (G - \nu). \label{eq:P}
	\end{array}\right.
\end{align}
\eqref{eq:px} and \eqref{eq:P} are well-posed.

Similarly, Theorem \ref{Thm:uvariate} gives a sufficient condition for $u^*$.
\begin{theorem}\label{Thm:uvariate}
	Suppose $G \succeq \nu$. For $t\in [0,T)$ and $s \in [t, T]$, define
	\begin{align*}
		\Lambda^x(s; t) &:= B_s p^x(s;t)+\sum_{j=1}^d(D_s^j)' k^{x,j}(s;t) + \zeta^*(s;t) R_s u^*_s - \zeta^*(s;t) \sum^d_{j=1}(h^*_{js})^2 \Xi_{2, j} u^*_s, \\
		\Sigma(s; t) &:= \zeta^*(s;t) R_s - \zeta^*(s;t) \sum^d_{j=1} (h^*_{js})^2 \Xi_{2, j}  + \sum_{j=1}^d (D_s^j)'P(s;t)D_s^j.
	\end{align*}
	If $u^*$ satisfies
	\begin{equation}\label{eq:u-suff}
		\Lambda^x(t; t) = 0, \quad \Sigma(t; t) \succeq 0, \quad \as, \aee \; t \in [0, T],
	\end{equation}
	then inequality \eqref{ineq:u*} holds for $u^*$.
\end{theorem}
The sufficient condition \eqref{eq:u-suff} is closely related to the optimality condition for the variational inequality in the stochastic maximum principle, as outlined in \cite[Theorem 3.2]{yz99}.

Compared with \cite{hjz12}, we require one more condition $G \succeq \nu$ to tackle with the difficulty that moment estimate $\ang{\nu \E^\p_t [\zeta^*(T;t) Y^x_T], \E^\p_t [\zeta^*(T;t) Y^x_T]} = O(\varepsilon)$ is not sufficient. In view of the boundary condition for $P(T; t)$, the assumption $G \succeq \nu$ is needed for $\Sigma(t; t) \succeq 0$. If we have an estimate of $o(\varepsilon)$ in special cases, we can define $P(s;t)$ with $G$ instead of $G - \nu$, as in the classical cases. The requirement $G \succeq \nu$ can then be removed.
\begin{corollary}\label{cor:special}
	Define $P(s;t)$ with $G$ instead of $G - \nu$ in \eqref{eq:P}. Suppose $C^j=0$. Then the claim in Theorem \ref{Thm:uvariate} holds without the assumption that $ G \succeq \nu$.
\end{corollary}

Before discussing specific examples, we explain the key difference between Theorem \ref{Thm:uvariate} and the case where preference ordering is imposed. The following explanation is heuristic and informal. Suppose the sufficient condition \eqref{eq:h-suff} for $h^*$ leads to an equilibrium $h^*_s = h^*(\tilde{X}^*_s, u^{t,\varepsilon,v}_s, s)$, where $\tilde{X}^*$ is defined as in Table \ref{tab:preforder}. In the case of preference ordering, the agent considers the state process \eqref{eq:aug-state} and the objective \eqref{eq:aug-obj}, with $h$ replaced by $h^*(X_s, u_s, s)$. We need the derivatives of $h^*(X_s, u_s, s)$ with respect to $X$ and $u$ when defining the adjoint processes and performing Taylor expansions of the cost functional. Following a procedure similar to \cite[Chapter 3.1]{yz99}, we expect that the corresponding first-order condition is given by
	\begin{equation}\label{eq:conjecture}
		B_t p^x(t;t) +\sum_{j=1}^d(D_t^j)' k^{x,j}(t;t) + h^*_u(X^*_t, u^*_t, t)' k^\zeta(t ;t)  + f^h_u (X^*_t, u^*_t, t) = 0,
	\end{equation}
where $f^h_u$ is the partial derivative of $f^h(X^*_t, u^*_t, t) := f(X^*_t, u^*_t, h^*(X^*_t, u^*_t, t), t)$ with respect to $u$. We omit the details on modifying the adjoint processes. To establish that \eqref{eq:conjecture} is a sufficient condition for the equilibrium $u^*$ with preference ordering, additional technical assumptions or specifications are required to ensure the existence of $h^*(\tilde{X}^*_s, u^{t,\varepsilon,v}_s, s)$ and its derivatives. This is beyond the scope of the present paper. Since $h^*$ can be adjusted according to $u$ when preference ordering is imposed, the worst-case scenario may be more severe than in our case without preference ordering.

\section{Examples and Well-posedness Results}\label{sec:examples}
In this section, we consider one dimensional process $X$, i.e. $n = 1$. Suppose the dimensions of Brownian motion and control $u$ are equal, i.e. $d = l$. We mainly consider the state-dependent risk aversion and thus set $\mu_2 = 0$. Let the ambiguity aversion (intercept) coefficient $\xi_{3, j} = 0$. Besides, assume parameters $b = 0, \sigma^j = 0$. We look for an equilibrium $(h^*, u^*)$ as defined in Definition \ref{def:equili}, by solving the sufficient conditions in Theorems \ref{Thm:hvariate} and \ref{Thm:uvariate}.

\subsection{State-dependent ambiguity aversion}
Consider the ambiguity aversion with a preference function depending on the process $X$ as $\Phi_j (X_s, u_s, s) = \xi_1 X^2_s$. For convenience, we summarize the process $X$ and $f$ in the objective functional again as follows:
\begin{align*}
	dX_s  &= (A_sX_s+B_s'u_s)ds + (C_s X_s + D_s u_s)'dW^\p_s, \\
	f(X_s, u_s, h_s, s) &= \frac{1}{2} \Big(Q_s X^2_s + \ang{ R_su_s, u_s} - \xi_1 h'_s h_s X^2_s \Big).
\end{align*}

To derive explicit solutions to the adjoint processes \eqref{pzeta} and \eqref{eq:px}, we are inspired by their terminal conditions and attempt with the following ansatz with $(X^*, \zeta^*(\cdot; t))$ corresponding to an equilibrium pair $(h^*, u^*)$:
\begin{align}
	p^\zeta(s; t) &= L_s (X^*_s)^2 - H_s X^*_s \E^\p_t [ \zeta^*(s;t) X^*_s] - F_s X^*_s X^*_t, \label{ans:pzeta} \\
	p^x(s;t) &= \zeta^*(s;t) M_s X^*_s - \zeta^*(s;t) N_s \E^\p_t [ \zeta^*(s;t) X^*_s] - \zeta^*(s;t) \Gamma_s X^*_t. \label{ans:px}
\end{align}
It\^o's lemma yields
\begin{align}
	d p^\zeta(s; t) =& [\dot{L}_s (X^*_s)^2 + 2 L_s X^*_s (A_s X^*_s + B'_s u^*_s) + L_s |C X^*_s + D_s u^*_s|^2]ds \nonumber \\ 
	& + 2 L_s X^*_s ( C X^*_s + D_s u^*_s)' dW^\p_s \label{eq:Ito-pzeta} \\ 
	& - \E^\p_t[\zeta^*(s;t) X^*_s] [\dot{H}_s X^*_s + H_s(A_s X^*_s + B'_s u^*_s)] ds \nonumber \\
	& - \E^\p_t[\zeta^*(s;t) X^*_s] H_s (C_s X^*_s + D_s u^*_s)' dW^\p_s \nonumber \\
	& - H_s X^*_s \E^\p_t[ \zeta^*(s; t) (A_s X^*_s + B'_s u^*_s)]ds \nonumber \\
	& - H_s X^*_s \E^\p_t[ \zeta^*(s; t) (C_s X^*_s + D_s u^*_s)'h^*_s]ds \nonumber \\
	& - \dot{F}_s X^*_s X^*_t ds - F_s X^*_t (A_s X^*_s + B'_s u^*_s)ds - F_s X^*_t(C_s X^*_s + D_s u^*_s)' dW^\p_s. \nonumber
\end{align}

Comparing with the definition of $p^\zeta(s;t)$ in \eqref{pzeta}, $dW^\p_s$ part implies
\begin{align*}
	k^\zeta(s;t) = & 2 L_s X^*_s ( C X^*_s + D_s u^*_s) - \E^\p_t[\zeta^*(s;t) X^*_s] H_s (C_s X^*_s + D_s u^*_s) \\
	& - F_s X^*_t(C_s X^*_s + D_s u^*_s).
\end{align*}

For $p^x(s;t)$, we have
\begin{align}
	d p^x(s;t) = & \big[\dot{M}_s \zeta^*(s;t) X^*_s + M_s \zeta^*(s;t) (A_s X^*_s + B'_s u^*_s)  \label{eq:Ito-px} \\
	& \;\; + M_s \zeta^*(s;t) (C_s X^*_s + D_s u^*_s)' h^*_s \big]ds \nonumber \\
	& + M_s \zeta^*(s;t) (C_s X^*_s + D_s u^*_s)' dW^\p_s + M_s \zeta^*(s;t) X^*_s (h^*_s)' dW^\p_s \nonumber \\
	& - N_s \zeta^*(s;t) \E^\p_t[ \zeta^*(s;t) (A_s X^*_s + B'_s u^*_s + (C_s X^*_s + D_s u^*_s)' h^*_s)]ds \nonumber \\
	& -  \dot{N}_s \zeta^*(s;t) \E^\p_t[ \zeta^*(s;t) X^*_s] ds -  N_s \zeta^*(s;t) \E^\p_t[ \zeta^*(s;t) X^*_s] (h^*_s)' dW^\p_s \nonumber \\
	& - X^*_t \dot{\Gamma}_s \zeta^*(s;t)ds - X^*_t \Gamma_s \zeta^*(s;t) (h^*_s)' dW^\p_s. \nonumber
\end{align}
The $dW^\p_s$ part implies that
\begin{align*}
	k^x(s;t) = &  \zeta^*(s;t) (C_s X^*_s + D_s u^*_s) M_s +  \zeta^*(s;t) M_s X^*_s h^*_s \\
	& -  N_s \zeta^*(s;t) \E^\p_t[ \zeta^*(s;t) X^*_s] h^*_s - \zeta^*(s;t) X^*_t \Gamma_s  h^*_s.
\end{align*}

In view of the sufficient conditions \eqref{eq:h-suff} and \eqref{eq:u-suff} for $h^*$ and $u^*$, we expect that $u^*_s = \alpha^*_s X^*_s$ and $h^*$ is deterministic. The sufficient condition for $h^*$ gives
\begin{equation*}
	\xi_1 (X^*_t)^2 h^*_t = k^\zeta(t;t) = (2 L_t - H_t - F_t) (C_t + D_t \alpha^*_t) (X^*_t)^2. 
\end{equation*}
Denote $E_t := 2 L_t - H_t - F_t$. Since $ X^*_t > 0$, then
\begin{equation}\label{eq:h*}
	h^*_t = \frac{E_t}{\xi_1}(C_t + D_t \alpha^*_t).
\end{equation}

The sufficient condition for $u^*$ leads to
\begin{align*}
	B_t p^x(t;t) + D'_t k^x(t;t) + R_t u^*_t =&  B_t( M_t - N_t - \Gamma_t) X^*_t + D'_t  (C_t + D_t \alpha^*_t) M_t X^*_t  \\
	& + D'_t h^*_t (M_t - N_t - \Gamma_t) X^*_t + R_t \alpha^*_t X^*_t \\
	=&  0.
\end{align*}
Denote $\Delta_t := M_t - N_t - \Gamma_t$. With the expression of $h^*_t$ in \eqref{eq:h*}, we simplify $\alpha^*$ as
\begin{equation}\label{eq:alpha*}
	\alpha^*_t = - [R_t + (M_t + \Delta_t E_t/\xi_1)D'_t D_t]^{-1} [(M_t + \Delta_t E_t/\xi_1)D'_t C_t + B_t \Delta_t].
\end{equation}

To obtain the ODE system for the coefficients in the ansatz, we compare the $ds$ parts in \eqref{pzeta} and \eqref{eq:px}, the definition of BSDEs, with the expressions in \eqref{eq:Ito-pzeta} and \eqref{eq:Ito-px}, derived from the ansatz.
The $ds$ part in \eqref{eq:Ito-pzeta} should equal to
\begin{align*}
	& - 2 L_s X^*_s ( C X^*_s + D_s u^*_s)' h^*_s - \E^\p_t[\zeta^*(s;t) X^*_s] H_s (C_s X^*_s + D_s u^*_s)' h^*_s \\
	& - F_s X^*_t(C_s X^*_s + D_s u^*_s)' h^*_s - \frac{1}{2} Q_s (X^*_s)^2 - \frac{1}{2} \ang{R_s u^*_s, u^*_s} + \frac{1}{2} \xi_1 (X^*_s)^2 |h^*_s|^2.
\end{align*}
Matching with the corresponding terms in \eqref{pzeta}, we have 
\begin{equation} \label{eq:LHF}
	\left\{
	\begin{aligned}
		\dot{L}_s =& - [2(A_s + B'_s \alpha^*_s) + 2 ( C_s + D_s \alpha^*_s)' h^*_s + | C_s + D_s \alpha^*_s|^2] L_s & \\
		& \quad -\frac{1}{2} Q_s - \frac{1}{2} \ang{R_s \alpha^*_s, \alpha^*_s} + \xi_1 \frac{|h^*_s|^2}{2}, & L_T = \frac{G}{2}, \\
		\dot{H}_s =& - [2(A_s + B'_s \alpha^*_s) + 2 (C_s + D_s \alpha^*_s)' h^*_s] H_s, & H_T = \nu, \\
		\dot{F}_s =& - [(A_s + B'_s \alpha^*_s) + (C_s + D_s \alpha^*_s)' h^*_s] F_s, & F_T = \mu_1.
	\end{aligned}
	\right.
\end{equation}
Similarly, matching the $ds$ part in \eqref{eq:px} and \eqref{eq:Ito-px}, we obtain
\begin{align}
	\dot{M}_s =& - [2A_s + B'_s \alpha^*_s + (2 C_s + D_s \alpha^*_s)' h^*_s + |C_s|^2 + C'_s D_s \alpha^*_s] M_s     \nonumber \\
	& - Q_s + \xi_1 |h^*_s|^2, \quad M_T = G, \nonumber \\
	\dot{N}_s =& - [2A_s + B'_s \alpha^*_s + (2C_s + D_s \alpha^*_s)' h^*_s] N_s, \quad N_T = \nu, \label{eq:MNGamma} \\
	\dot{\Gamma}_s =& - [A_s + C'_s h^*_s] \Gamma_s, \quad \Gamma_T = \mu_1. \nonumber
\end{align}

With $h^*$ in \eqref{eq:h*} and $\alpha^*$ in \eqref{eq:alpha*}, the ODE systems \eqref{eq:LHF} and \eqref{eq:MNGamma} are coupled and should be viewed as a system with six variables $(L, H, F, M, N, \Gamma)$. It is far complicated than the classical Riccati-type equations in TC or TIC LQ controls \cite{yz99,lz02,hjz12}. Indeed, take the term $(\alpha^*_s)'h^*_s L_s$ as an example. It has terms such as $\Delta^2_s E^3_s L_s$, which poses great challenges to the proof of the well-posedness of \eqref{eq:LHF} and \eqref{eq:MNGamma}. Briefly speaking, we use the truncation method to validate that the solutions $(L, H, F, M, N, \Gamma)$ are positive under certain conditions. Moreover, we anticipate that $\Delta \leq 0$ and $E \leq 0$.

In the following lemma and its proof, we adopt the convention that notations with an underline (resp. upper line), e.g. $\ubar{m}$ (resp. $\bar{m}$), mean the lower (resp. upper) bound.
\begin{lemma}\label{Lem:StateODE}
	Consider constants $\ubar{m}$, $\bar{m}$, $\ubar{\delta}$, and $\ubar{e}$ satisfying
	\begin{align*}
		&0 < \ubar{m} \leq G \leq \bar{m}, \quad \ubar{\delta} \leq G - \nu - \mu_1 \leq 0, \quad \ubar{e} \leq G - \nu - \mu_1 \leq 0.
	\end{align*}
	For any $t \in [0, T]$, suppose the following conditions hold:
	\begin{align}
		&\ubar{L}(t) \geq 0, \quad \ubar{M}(t) \geq \ubar{m}, \quad \bar{M}(t) \leq \bar{m}, \nonumber \\
		& \bar{M}(t) - \ubar{N}(t) - \ubar{\Gamma}(t) \leq 0, \quad \ubar{M}(t) - \bar{N}(t) - \bar{\Gamma}(t) \geq \ubar{\delta}, \label{cond:StateODE} \\
		& 2 \bar{L}(t) - \ubar{H}(t) - \ubar{F}(t) \leq 0, \quad 2\ubar{L}(t) - \bar{H}(t) - \bar{F}(t) \geq \ubar{e}, \nonumber
	\end{align}
	with functions $\ubar{L}(t)$, $\ubar{M}(t)$, $\bar{M}(t)$, $\ubar{N}(t)$, $\bar{N}(t)$, $\ubar{\Gamma}(t)$, $\bar{\Gamma}(t)$, $\ubar{H}(t)$, $\bar{H}(t)$, $\ubar{F}(t)$, and $\bar{F}(t)$ defined in the proof, which are independent of solutions to \eqref{eq:LHF}-\eqref{eq:MNGamma}. Moreover, assume
	\begin{equation}
		R_t + \ubar{m} D'_t D_t \succeq \tau I, \; t \in [0, T]
	\end{equation}
	for some $\tau > 0$.

	Then the ODE system \eqref{eq:LHF}-\eqref{eq:MNGamma} admits a unique non-negative solution satisfying
	\begin{align*}
		&L(t) \geq 0, \quad \ubar{m} \leq M(t) \leq \bar{m}, \quad \ubar{\delta} \leq \Delta(t) \leq 0, \quad \ubar{e} \leq E(t) \leq 0, \quad t \in [0,T].
	\end{align*}
\end{lemma}

Since $h^*$ is deterministic, the solution to $P(s;t)$ in \eqref{eq:P} is given by 
\begin{align*}
	P(s;t) =& \zeta^*(s;t) (G-\nu) e^{\int^T_s (2A_u + 2 C'_uh^*_u + |C_u|^2)du} \\
	& + \zeta^*(s;t) \int^T_s e^{\int^v_s (2A_u + 2 C'_uh^*_u + |C_u|^2)du} (Q_v - \xi_1 |h^*_v|^2) dv.
\end{align*}
Proposition \ref{Cor:StateAdm} shows that $(h^*, u^*)$ is admissible and an equilibrium.
\begin{proposition}\label{Cor:StateAdm}
	Suppose 
	\begin{itemize}
		\item[(1)] assumptions in Lemma \ref{Lem:StateODE} hold;
		\item[(2)] $R_t + D'_t P(t;t) D_t \succeq 0$ with $P(\cdot;t)$ satisfying conditions in Theorem \ref{Thm:uvariate} or Corollary \ref{cor:special}.
	\end{itemize}
	Then $(h^*, u^*)$ in \eqref{eq:h*}-\eqref{eq:alpha*} is admissible and an equilibrium.
\end{proposition}

For a small time horizon $T$, we provide an example that satisfies the assumptions in Lemma \ref{Lem:StateODE} and Proposition \ref{Cor:StateAdm}. Motivated by the MV portfolio selection problem, we consider the following constant parameters:
\begin{align*}
	G = 1, \quad \nu = 1, \quad D =1, \quad R = 0, \quad C =0, \quad Q = 0, \quad \mu_1 > 0.
\end{align*}
No additional assumptions are made on $A_t$ and $B_t$ beyond the standing Assumption \ref{assum:state}.

For a small $\varepsilon_0$ satisfying $0 < \varepsilon_0 \leq \min\{1/2, \mu_1 \}$, the continuity property of $P(t; t)$ and the functions in \eqref{cond:StateODE}, along with the boundary conditions, imply the existence of a small $T := \delta_0 > 0$, such that for $t \in [0, T]$, we have:
	\begin{align*}
		&\ubar{L}(t) \in [1/2 - \varepsilon_0, 1/2 + \varepsilon_0], \quad \ubar{M}(t) \in [1 - \varepsilon_0, 1 + \varepsilon_0], \quad \bar{M}(t) \in [1 - \varepsilon_0, 1 + \varepsilon_0], \\
		& \bar{M}(t) - \ubar{N}(t) - \ubar{\Gamma}(t) \in [- \mu_1 - \varepsilon_0, - \mu_1 + \varepsilon_0], \\
		& \ubar{M}(t) - \bar{N}(t) - \bar{\Gamma}(t) \in [- \mu_1 - \varepsilon_0, - \mu_1 + \varepsilon_0], \\
		& 2 \bar{L}(t) - \ubar{H}(t) - \ubar{F}(t) \in [- \mu_1 - \varepsilon_0, - \mu_1 + \varepsilon_0], \\
		& 2\ubar{L}(t) - \bar{H}(t) - \bar{F}(t) \in [- \mu_1 - \varepsilon_0, - \mu_1 + \varepsilon_0], \\
		& P(t;t) \in [1 - \varepsilon_0, 1 + \varepsilon_0].
	\end{align*}
Here, $P(\cdot; t)$ is defined with $G$ because $C = 0$, as noted in Corollary \ref{cor:special}.
	
Let $ \ubar{m} = 1 - \varepsilon_0$, $\bar{m} = 1 + \varepsilon_0$, $\ubar{e} = \ubar{\delta} = - \mu_1 - \varepsilon_0$. Under these conditions, \eqref{cond:StateODE} and the other assumptions in Lemma \ref{Lem:StateODE} and Proposition \ref{Cor:StateAdm} are satisfied.

\subsection{Control-dependent ambiguity aversion}
In this section, we consider a control-dependent ambiguity aversion as $\Phi_j(X_s, u_s, s) = \xi_2 l u^2_{js}$. For the well-posedness of the problem, we should consider the equilibrium strategy satisfying $u^2_{js} > 0$, $\as, \, \aee \, t \in [0, T]$. Otherwise, there is no penalty for deviating from the reference measure. This condition can be verified a posteriori with the derived equilibrium $u^*$.

Similar to the state-dependent ambiguity aversion, we still suppose $\mu_2, b, \sigma^j = 0$, as well as $C^j = 0$ for tractability. To ease the burden of presentation, we consider $D_t$ and $R_t$ as identity matrices multiplied by scalar functions which are still denoted as $D_t$ and $R_t$ for simplicity.

We conjecture that $p^\zeta(s;t)$ and $p^x(s;t)$ have the same form as in \eqref{ans:pzeta}-\eqref{ans:px}. Without confusion, we use the same notations for the ansatz. Try the candidate equilibrium $u^*_t = \alpha^*_t X^*_t$. The sufficient condition \eqref{eq:h-suff} for $h^*$ gives
\begin{equation*}
	\xi_2 l (\alpha^*_{jt})^2 h^*_{jt} = E_t D_t \alpha^*_{jt}.
\end{equation*}
Since $\alpha^*_{jt} \neq 0$, then
\begin{equation}\label{eq:ctrl_h*}
	h^*_{jt} = \frac{E_t D_t}{\xi_2 l \alpha^*_{jt}}.
\end{equation}
The sufficient condition for $u^*$ is
\begin{align*}
	& B_t p^x(t;t) + D_t k^x(t;t) + R_t u^*_t - \xi_2 l h^*_t  \odot h^*_t \odot u^*_t = 0.
\end{align*}

After simplification, it reduces to
\begin{equation}\label{eq:alpha_nonzero}
	(D^2_t M_t + R_t) \alpha^*_{jt} + B_{jt} \Delta_t + \frac{D^2_t(\Delta_t E_t - E^2_t)}{\xi_2 l \alpha^*_{jt}} = 0.
\end{equation}
Since $\alpha^*_{jt}$ is non-zero, we multiply $\alpha^*_{jt}$ on both sides and denote the equation as
\begin{equation}\label{eq:quad-alpha}
	\kappa_{t} (\alpha^*_{jt})^2 + \beta_{jt} \alpha^*_{jt} + \gamma_t = 0, 
\end{equation}
with corresponding $\kappa_{t}$, $\beta_{jt}$, and $\gamma_t$. Note that $\Delta_T = E_T$. One root of the quadratic equation \eqref{eq:quad-alpha} always leads to $\alpha^*_{jT} = 0$. Thus this root is redundant. Finally, only one equilibrium is solved from \eqref{eq:quad-alpha}. 

Similarly, one can derive the following ODE system, which is simpler than the previous ODE system \eqref{eq:LHF}--\eqref{eq:MNGamma} for the state-dependent ambiguity aversion.
\begin{eqnarray}\label{eq:CtrlODE}
	\left\{\begin{array}{lcl}
		\dot{L}_s &=& -[2A_s + 2B'_s \alpha^*_s + 2 D^2_s E_s/\xi_2 + D^2_s |\alpha^*_s|^2] L_s - Q_s/2 \\
		& & - R_s |\alpha^*_s|^2/2 + E^2_s D^2_s /(2\xi_2), \\
		\dot{H}_s &=& -[2A_s + 2B'_s \alpha^*_s + 2 D^2_s E_s/\xi_2 ] H_s, \\
		\dot{F}_s &=& -[A_s + B'_s \alpha^*_s + D^2_s E_s/\xi_2 ] F_s, \\
		\dot{M}_s &=& -[2A_s + B'_s \alpha^*_s + D^2_s E_s/\xi_2 ] M_s - Q_s, \\
		\dot{N}_s &=& -[2A_s + B'_s \alpha^*_s + D^2_s E_s/\xi_2 ] N_s, \\
		\dot{\Gamma}_s &=& - A_s \Gamma_s.
	\end{array}\right.
\end{eqnarray}

Note that $\Gamma$ is explicitly solvable and we only need to prove the existence and uniqueness for other five variables.

\begin{lemma}\label{Lem:CtrlODE}	
	Consider constants $\ubar{m}$, $\bar{m}$, $\bar{\delta}$, and $\bar{e}$ satisfying
	\begin{align*}
		& 0 < \ubar{m} \leq G \leq \bar{m}, \quad - \bar{\delta} \leq G - \nu - \mu_1 \leq \bar{\delta}, \quad -\bar{e} \leq G - \nu - \mu_1 \leq \bar{e}.
	\end{align*}
	For any $t \in [0, T]$, suppose the following conditions hold:
	\begin{align}
		&\ubar{L}(t) \geq 0, \quad \ubar{M}(t) \geq \ubar{m}, \quad \bar{M}(t) \leq \bar{m}, \quad \bar{I}(t) - \Gamma(t) \leq \bar{\delta}, \quad \ubar{I}(t) -\Gamma(t) \geq -\bar{\delta}, \label{cond:CtrlODE} \\
		& 2 \bar{L}(t) - \ubar{H}(t) - \ubar{F}(t) \leq \bar{e}, \quad 2\ubar{L}(t) - \bar{H}(t) - \bar{F}(t) \geq -\bar{e}.  \nonumber
	\end{align}
	with functions $\ubar{L}(t)$, $\bar{L}(t)$, $\ubar{M}(t)$, $\bar{M}(t)$, $\ubar{I}(t)$, $\bar{I}(t)$, $\ubar{H}(t)$, $\bar{H}(t)$, $\ubar{F}(t)$, and $\bar{F}(t)$ defined in the proof, which are independent of solutions to \eqref{eq:CtrlODE}. Moreover, assume
	\begin{equation}\label{eq:discriminant}
		\xi_2 l B^2_{jt} - D^2_t (D^2_t \bar{m} + R_t) \geq 0, \; t \in [0, T], \; j = 1, ... , l.
	\end{equation}	
	Then the ODE system \eqref{eq:CtrlODE} admits a unique non-negative solution satisfying
	\begin{align*}
		&L(t) \geq 0, \quad \ubar{m} \leq M(t) \leq \bar{m}, \quad -\bar{\delta} \leq \Delta(t) \leq \bar{\delta}, \quad -\bar{e} \leq E(t) \leq \bar{e}, \quad t \in [0,T].
	\end{align*}	
\end{lemma}

Proposition \ref{Cor:CtrlAdm} shows that $(h^*, u^*)$ is admissible and an equilibrium.
\begin{proposition}\label{Cor:CtrlAdm}
	Suppose 
	\begin{itemize}
		\item[(1)] assumptions in Lemma \ref{Lem:CtrlODE} hold;
		\item[(2)] $R_t - \xi_2 l (h^*_{jt})^2 + D'_t P(t;t) D_t \succeq 0, j = 1, ..., l$, with $P(\cdot;t)$ satisfying conditions in Theorem \ref{Thm:uvariate} or Corollary \ref{cor:special};
		\item[(3)] $|\alpha^*_{jt}| \geq \phi > 0$ for $j = 1, ..., l$, $ t \in [0, T]$.
	\end{itemize}
	Then $(h^*, u^*)$ in \eqref{eq:ctrl_h*}-\eqref{eq:quad-alpha} is admissible and an equilibrium.
\end{proposition}

We now provide an example that satisfies the assumptions in Lemma \ref{Lem:CtrlODE} and Proposition \ref{Cor:CtrlAdm}. Consider the constant parameters given by
\begin{align*}
	G = 1, \quad \nu = 1, \quad D =1, \quad R = 0, \quad C =0, \quad Q = 0, \quad \mu_1 > 0.
\end{align*}
Additionally, assume that $A_t > 0$ and $B_t \equiv B$, with $l\xi_2 B^2 > 1$. 
	
Recall that $\Gamma_t = \mu_1 e^{\int^T_t A_s ds}$. Under the current specification, we have $\ubar{I}(t) = \bar{I}(t) = 0$. For a small $\varepsilon_0$ such that $0 < \varepsilon_0 \leq \min\{1/2, 1 - 1/(l\xi_2 B^2), \mu_1 |B|/2, l\xi_2 B^2 - 1 \}$, the continuity property of $P(t; t)$ and the functions in \eqref{cond:CtrlODE}, along with the boundary conditions, imply that there exists a small $T := \delta_0 > 0$, such that for $t \in [0, T]$, the following conditions hold:
	\begin{align*}
		&\ubar{L}(t) \in [1/2 - \varepsilon_0, 1/2 + \varepsilon_0], \quad \ubar{M}(t) \in [1 - \varepsilon_0, 1 + \varepsilon_0], \quad \bar{M}(t) \in [1 - \varepsilon_0, 1 + \varepsilon_0], \\
		& 2 \bar{L}(t) - \ubar{H}(t) - \ubar{F}(t) \in [- \mu_1 - \varepsilon_0, - \mu_1 + \varepsilon_0], \\
		& 2\ubar{L}(t) - \bar{H}(t) - \bar{F}(t) \in [- \mu_1 - \varepsilon_0, - \mu_1 + \varepsilon_0], \\
		& P(t;t) \in [1 - \varepsilon_0, 1 + \varepsilon_0].
	\end{align*}
Here, by Corollary \ref{cor:special}, $P(\cdot;t)$ is also defined with $G$ since $C=0$.
	
Let $ \ubar{m} = 1 - \varepsilon_0$, $\bar{m} = 1 + \varepsilon_0$, $ \bar{\delta} \geq \mu_1 e^{\int^T_0 A_s ds}$, and $ \bar{e} = \mu_1 + \varepsilon_0$. Then conditions \eqref{cond:CtrlODE} and \eqref{eq:discriminant} in Lemma \ref{Lem:CtrlODE} are valid. Conditions (2) and (3) in Proposition \ref{Cor:CtrlAdm} hold because $h^*_{jt}$ and $\alpha^*_{jt}$ are continuous and $h^*_{jT} = -1/(\xi_2 l B)$, $\alpha^*_{jT} = \mu_1 B$.

\section{Numerical Studies}\label{sec:num}
A relevant application of the general LQ control theory is the MV portfolio selection problem \cite{lz02,bc10,hjz12,bkm17,hpw21FS}. Consider a complete market with one risk-free asset, with a risk-free rate $A_t > 0$, and $d$ risky assets. The price vector $S$ of the risky assets is driven by the following dynamics:
	\begin{equation}\label{bsstate}
		dS_t = \hbox{diag}(S_t) \big( m_t dt + \sigma_t dW^\p_t \big),
	\end{equation}
	where $\hbox{diag}(S_t)$ is the diagonal matrix with the elements of $S_t$ and $m_t$ and $\sigma_t$ represent the continuous return vector and volatility matrix of the risky assets, respectively. We assume that $\sigma_t$ is invertible.
	
	Let $X_s$ denote wealth and $\pi_s X_s$ represent the amount invested in the $d$ risky assets at time $s$. By the self-financing principle, the wealth process $X$ evolves according to:
	\begin{equation*} 
		dX_s = \big( A_s X_s+(m_s-A_s\id)'\pi_s X_s \big)ds + X_s\pi_s'\sigma_s dW^\p_s=(A_sX_s + B'_s \alpha_s X_s)ds + X_s \alpha_s'dW^\p_s,
	\end{equation*}
	where $\id := (1,\ldots,1)'\in\R^d$, $\alpha_s = \sigma_s'\pi_s$ is the transformed control vector, and $B_s = \sigma_s^{-1}(m_s-A_s\id)$ is the Sharpe ratio vector.
	
Under the alternative measure corresponding to $h^*$, the wealth process $X^*$ follows:
\begin{equation}
	dX^*_s = [A_sX^*_s+ (B_s + h^*_s)'\alpha^*_s X^*_s]ds + X^*_s (\alpha^*_s)' dW^*_s,
\end{equation}
where $W^*$ is the $d$-dimensional standard Brownian motion under the measure corresponding to $h^*$. 

The objective function of MV investors with ambiguity aversion is
\begin{equation*}
	\frac{1}{2}\left[\E^*_t [(X^*_T)^2]  - (\E^*_t [X^*_T])^2\right] - \mu_1 x_t \E^*_t [X^*_T] -  \frac{1}{2} \E^*_t \left[ \int_t^T (h^*_s)'\Phi(X^*_s, u^*_s, s)h^*_s ds \right],
\end{equation*}
where $\E^*_t[\cdot]$ is the conditional expectation under the measure corresponding to $h^*$. The state-dependent (resp. control-dependent) ambiguity aversion preference $\Phi(X^*_s, u^*_s, s)$ are chosen as $\xi_1 (X^*_s)^2$ (resp. $\xi_2 l (u^*_{js})^2$). It is noteworthy that both choices are aligned with the dimensional analysis as in \cite[Section 4.1]{bmz14} and \cite{hpw21FS}. Specifically, both the variance operator and the state-dependent expectation have a dimension of $(\text{dollar})^2$, and thus a reasonable ambiguity aversion function should agree with the same dimension.

In our numerical studies, we mainly investigate the impact of ambiguity aversion and risk aversion on the expected return and risk (variance). Note that $A, B, \alpha^*, h^*$ are deterministic. A direct calculation gives 
\begin{align*}
	\E^*[X^*_T] & = x_0 e^{\int^T_0 (A_s + (B_s + h^*_s)'\alpha^*_s) ds}, \\
	\text{Var}^* [X^*_T] &= x^2_0 e^{\int^T_0 2 (A_s + (B_s + h^*_s)'\alpha^*_s) ds} \left( e^{\int^T_0 |\alpha^*_s|^2 ds} -  1 \right),
\end{align*}
where $x_0$ is the initial wealth.

In the following subsections, we consider a benchmark setting with a risk-free rate $A = 0.02$ and a risky asset with risk premium $B = 0.375$ as in \cite{m04}. Initial wealth $x_0 =1$ and investment horizon $T = 1.0$. We focus on the closed-loop state-dependent ambiguity aversion in \cite{hpw21FS}, open-loop state-dependent ambiguity aversion, and open-loop control-dependent ambiguity aversion (CDAA). Without confusion, $\xi$ is interpreted as the ambiguity aversion parameter $\xi_1$ or $\xi_2$ in this paper. Note that \cite{hpw21FS} used $1/\xi$ as the aversion parameter and we have converted it to our formulation. It is difficult to obtain explicit solutions for the complicated ODE system. However, the following asymptotic analysis shows a complex interaction between risk aversion and ambiguity aversion. Crucially, ambiguity aversion demonstrates a stronger impact on investment decisions than the expected utility optimization setting \cite{m04}. Our code is available at \url{https://github.com/hanbingyan/DiffGame}.

\subsection{Varying Risk Aversion}\label{sec:risk}
First, we vary the risk aversion only and fix the ambiguity aversion at a given level of $\xi=10$. Figure \ref{fig:mu} depicts the evolution of the mean and standard deviation pairs of the equilibrium wealth process. To understand the loops in the figure, we note the following:
\begin{enumerate}[leftmargin=2em]
	\item The case $\mu_1 = 0$ is the most risk-averse. An equilibrium control pair is $h^* = 0$ and $u^* = 0$ for all the three cases. Investors who are extremely risk averse only invest in the risk-free asset. The risk premium under the worst-case scenario in equilibrium is $B + h^* = B$.
	\item The case $\mu_1 \rightarrow \infty$ corresponds to the absence of risk aversion. Under state-dependent ambiguity aversion, an equilibrium control pair for both open-loop and closed-loop settings is given by $h^* = -B$ and $u^* = 0$. 
	
	We compare this with the case where $\mu_1 \rightarrow \infty$ but ambiguity aversion is absent. To present more general results, we consider a continuous, deterministic risk-free rate $A_t$ and risk premium $B_t$. In the absence of ambiguity aversion, \cite{hjz12} derives the open-loop equilibrium strategy as
		\begin{equation}\label{eq:open_norob}
			\alpha^*_t = \frac{\mu_1 e^{-\int^T_t A_vdv}}{1 + \mu_1 \int^T_t e^{-\int^T_s A_vdv} |B_s|^2 ds} B_t. 
		\end{equation}
		\cite{bmz14} shows that the closed-loop equilibrium strategy is
		\begin{equation}\label{eq:closed_norob}
			\alpha^*_t = - B_t + B_t e^{ -\int^T_t |\alpha^*_s|^2 ds}  + \mu_1 B_t e^{\int^T_t (- A_s - B'_s \alpha^*_s - |\alpha^*_s|^2 )ds}. 
		\end{equation}
		
		Assuming $B_t > 0$ for all $t \in [0, T]$, both expressions yield $\alpha^*_T = \mu_1 B_T \rightarrow \infty$. For $t \in [0, T)$, the open-loop strategy in \eqref{eq:open_norob} converges to a finite, nonzero value. For the closed-loop strategy in \eqref{eq:closed_norob}, the limiting behavior is less transparent analytically but numerical evidence suggests that it also converges to a finite, nonzero value for large $\mu_1$.
		
		Furthermore, in both cases, we observe the divergence of the expected terminal wealth without ambiguity aversion and its variance:
		\begin{equation}\label{eq:diverge}
			\E^\p[X^*_T] \rightarrow \infty, \quad \text{Var}^\p [X^*_T] \rightarrow \infty, \quad \text{ when } \mu_1 \rightarrow \infty,
		\end{equation}
		where the expectations are taken under the model without ambiguity aversion. The claim in \eqref{eq:diverge} is provable for the open-loop case and supported numerically for the closed-loop case.
		
		In contrast, as long as the ambiguity aversion coefficient is positive and finite, the investor invests zero in the risky asset for any $t \in [0, T]$ when $\mu_1 \rightarrow \infty$. Ambiguity aversion also makes the risk premium under the worst case zero: $B + h^* = 0$.
		
		In the open-loop CDAA setting, the equilibrium obtained from \eqref{eq:alpha_nonzero} results in $\alpha^*_{T} = \mu_1 B_T$, which is the same terminal value as in \eqref{eq:open_norob} and \eqref{eq:closed_norob}.

		\item Since $\mu_1 = 0$ and $\mu_1 \rightarrow \infty$ lead to the same investment decision under state-dependent ambiguity aversion, the standard deviation and mean curves generate circles under both open-loop and closed-loop settings. Then there exist certain risk-aversion levels achieving the highest mean or standard deviation.   
	\end{enumerate}
	The limiting behaviors of the state-dependent case can be formally proved as follows.
	\begin{corollary}\label{Cor:limit}
		With a given finite $\xi_1 > 0$, both closed-loop and open-loop state-dependent ambiguity aversion have an equilibrium pair given by
		\begin{itemize}[label=\raisebox{0.25ex}{\tiny$\bullet$}]
			\item $h^* = 0$ and $u^* = 0$, when $\mu_1 = 0$;
			\item $h^* = - B$ and $u^* = 0$, when $\mu_1 \rightarrow \infty$.
		\end{itemize} 
	\end{corollary}
	
	Similarly, the control-dependent case for $\mu_1 = 0$ can be proven, while the case with $\mu_1 \rightarrow \infty$ is studied only through numerical methods. 
	\begin{corollary}\label{Cor:Ctrl}
		Suppose $\xi_1 > 0$ and $\mu_1 = 0$. In the open-loop control-dependent ambiguity aversion case, an equilibrium pair is given by $h^* = 0$ and $u^* = 0$.
	\end{corollary}
	It is important to note that the CDAA framework assumes $\alpha^*_t \neq 0$ in order to impose penalties on $h^*$. Thus, Corollary \ref{Cor:Ctrl} should be interpreted as a degenerate case.
	
	Note that the MV investment proportion $\alpha^*$ is given by \eqref{eq:mv-alpha*} for the state-dependent case:
	\begin{equation}\label{eq:mv-alpha*}
		\alpha^*_t = - [M_t + \Delta_t E_t/\xi_1]^{-1}  B_t \Delta_t.
	\end{equation}
	A closer look at $\alpha^*$ in \eqref{eq:mv-alpha*} gives some economic insights into the effects of ambiguity aversion. From the boundary condition $M_T = G$, one can interpret that $M_t$ comes from the variance operator and it usually reduces the investment amount $\alpha^*_t$. The dominant term of $\Delta_t$ and $E_t$ is $-\mu_1$ when $\mu_1$ is large enough. Therefore, when the investor becomes less risk averse, i.e. $\mu_1$ is larger, $-\Delta_t$ in the numerator enlarges the investment amount. However, ambiguity aversion represented by a positive $\xi_1 > 0$, appears in the denominator and shrinks the investment in a quadratic speed of risk aversion $\mu_1$ from $\Delta_t E_t$. Therefore, ambiguity aversion reduces the stock position at a faster rate than the increment from the numerator. This new behavior reflects that ambiguity aversion alters the {\it effective} risk aversion faster than a linear growth rate, which results in the zero investment strategy when $\mu_1 \rightarrow \infty$. The impact of ambiguity aversion is significant, as an ambiguity-averse investor with low risk aversion may allocate a similar proportion of their portfolio to stocks as an ambiguity-neutral investor with high risk aversion. In practice, investors often invest less aggressively in stocks than classical portfolio theory suggests. While \cite{m04} attributes this behavior to either excessive ambiguity aversion or excessive risk aversion, our work demonstrates that rational investors may adopt a conservative investment strategy even with moderate levels of both ambiguity and risk aversion.
	
	The state-dependent ambiguity aversion could be conservative, especially for the less risk-averse investors. When $\mu_1 \rightarrow \infty$, one may want to avoid the zero-investment strategy and recover the non-robust limiting behavior even when ambiguity exists. Then our CDAA specification becomes a potential candidate since it behaves closer to the non-robust case. It offers another perspective on which ambiguity aversion preference investors should choose.   
	
	\subsection{Varying Ambiguity Aversion}
	
	\begin{figure}
		\centering
		\begin{minipage}{.5\textwidth}
			\centering
			\includegraphics[width=.9\linewidth]{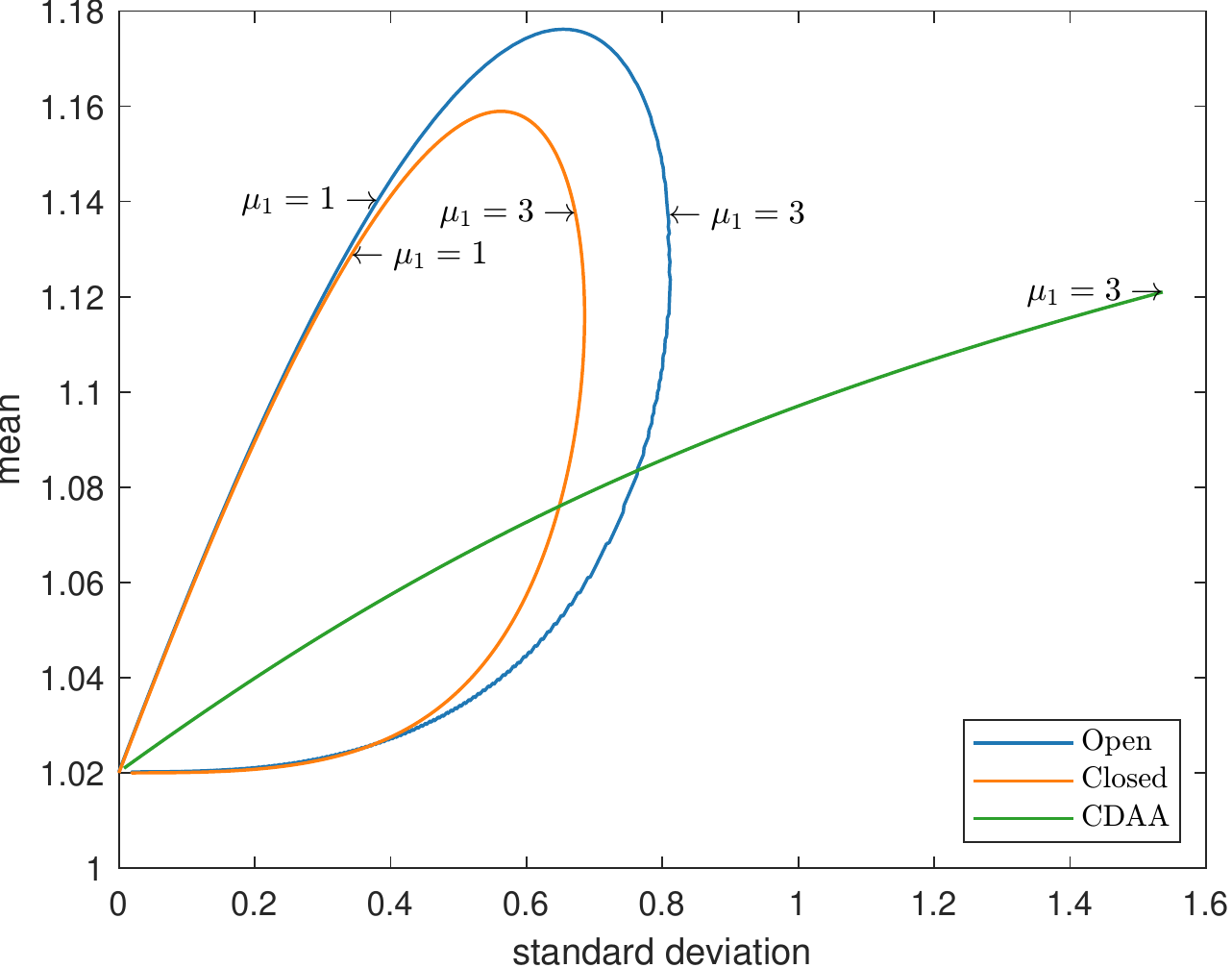}
			\captionof{figure}{Varying risk aversion}
			\label{fig:mu}
		\end{minipage}%
		\begin{minipage}{.5\textwidth}
			\centering
			\includegraphics[width=.9\linewidth]{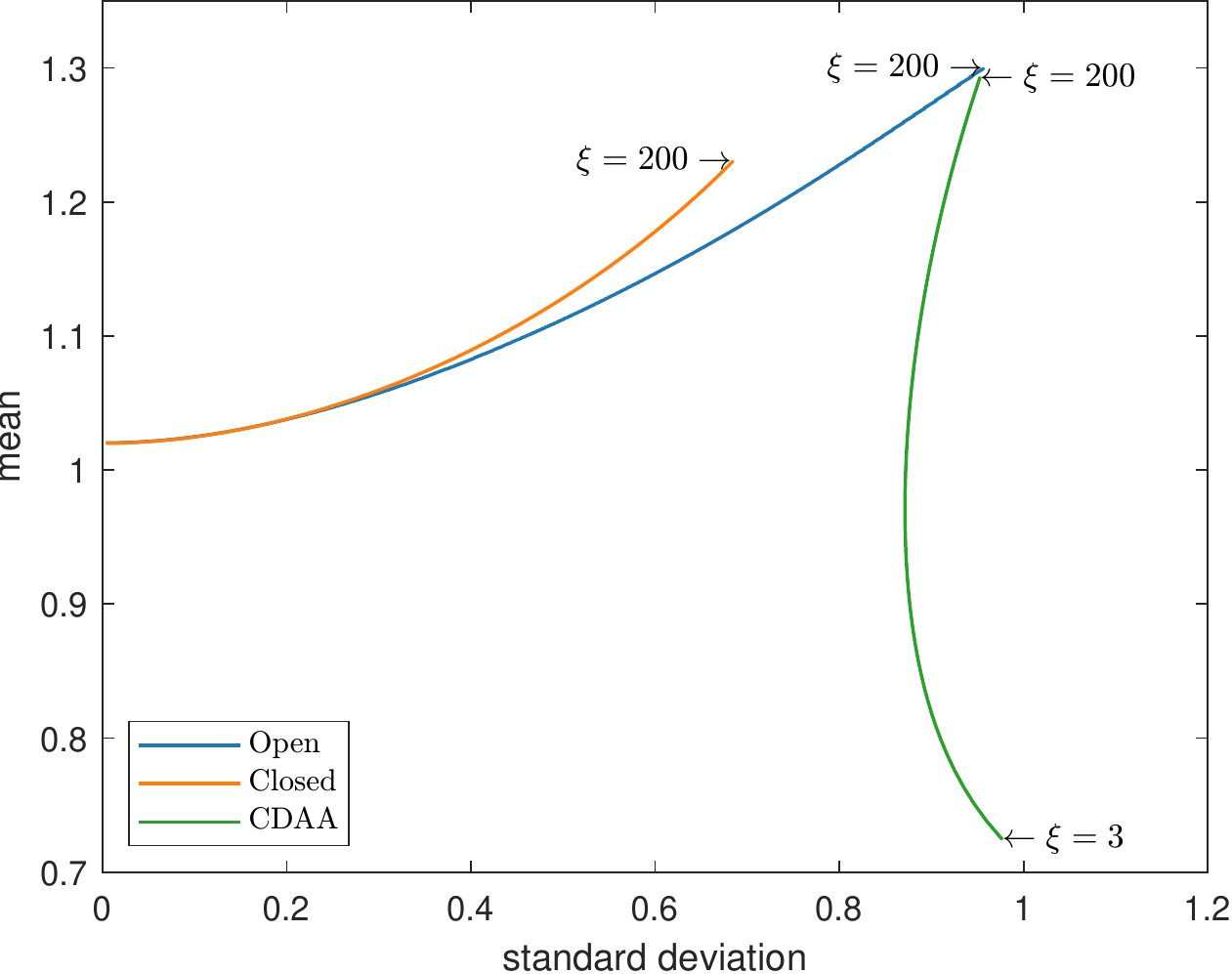}
			\captionof{figure}{Varying ambiguity aversion}
			\label{fig:xi}
		\end{minipage}
	\end{figure}

	If we fix the risk aversion $\mu_1 = 2$ and vary the ambiguity aversion, then Figure \ref{fig:xi} shows that
	\begin{enumerate}
		\item $\xi = 0$ is the most ambiguity-averse case. An equilibrium control pair is $h^* = -B$ and $u^* = 0$ for the state-dependent specification. The proof is very similar to that of Corollary \ref{Cor:limit} and thus omitted. The most ambiguity-averse investor believes the risk premium under the worst-case scenario is zero and thus does not invest in the risky asset. For the CDAA specification, the $h^*$ in \eqref{eq:ctrl_h*} has excluded the case with $u^*=0$. When $\xi \rightarrow 0$, the solution is numerically unstable as $h^*$ in \eqref{eq:ctrl_h*} is dividing $\xi$. Notably, the anticipated returns are even lower than the risk-free rate under the CDAA specification with a high ambiguity aversion.
		\item $\xi \rightarrow \infty$ is the case without ambiguity aversion. The state-dependent cases reduce to the solutions in \cite{bmz14,hjz12}. Since the CDAA formulation is open-loop, it converges to the same open-loop solution in \cite{hjz12}. The proof follows straightforwardly from the ODE in \eqref{eq:CtrlODE} as $\xi \rightarrow \infty$.
	\end{enumerate}
	
	\cite{m04} found the investment decision with power utility and ambiguity aversion is observationally equivalent to the non-robust case with a higher risk aversion. Loosely speaking, \cite{m04} showed the robust strategy with risk aversion $\mu_1$ and ambiguity aversion $\xi$ is the same as the non-robust one with risk aversion $\mu_1 + \xi$. Our results in these two subsections indicate that the effect of ambiguity aversion is not merely additive to risk aversion. The investor with the same $\mu_1 + \xi$ can still have different equilibrium strategies for different pairs of $(\mu_1, \xi)$.
	
	On a side note, \cite{m04} scaled the ambiguity aversion with the risk aversion or equivalently, use $\mu_1 \xi'$ in the place of $\xi$ with some constant $\xi'$. \cite[Footnote 9]{m04} explained this formulation for two reasons. First, it maintains the positivity of the preference since $1 - \gamma$ in the power utility may have different signs. Second, it recovers the logarithmic utility when $\gamma \rightarrow 1$. However, these two reasons do not apply to our MV framework. Therefore, we set the preference parameters without scaling. Nevertheless, it is crucial to investigate the relationship between ambiguity aversion and risk aversion. In the next subsection, we consider linear and quadratic relationships to motivate discussions on this underexplored topic in the literature.
	
	\subsection{Discussion: Relations between Risk and Ambiguity Aversions}
		
	The case where $\mu_1 \to \infty$ in Section \ref{sec:risk} prompts the question: if we also let $\xi \to \infty$ simultaneously, at what rate can we recover behavior similar to \eqref{eq:diverge}? Our framework allows for a flexible relationship between ambiguity and risk aversion. While analytical analysis is challenging, we can explore this problem numerically.
		
	We begin by considering a linear specification, $\xi = 5.0 + 2.0\mu_1$, which yields numerically stable results under the current parameter settings. As shown in Figure \ref{fig:linear}, the state-dependent cases converge to finite mean values and standard deviations. This convergence is verified numerically, as further increases in the aversion parameters do not noticeably alter the mean or standard deviation. Experimentally, the limit is highly dependent on the path taken to reach it, and consequently, on the coefficients of the linear relationship.

	\begin{figure}
			\centering
			\includegraphics[width=0.4\linewidth]{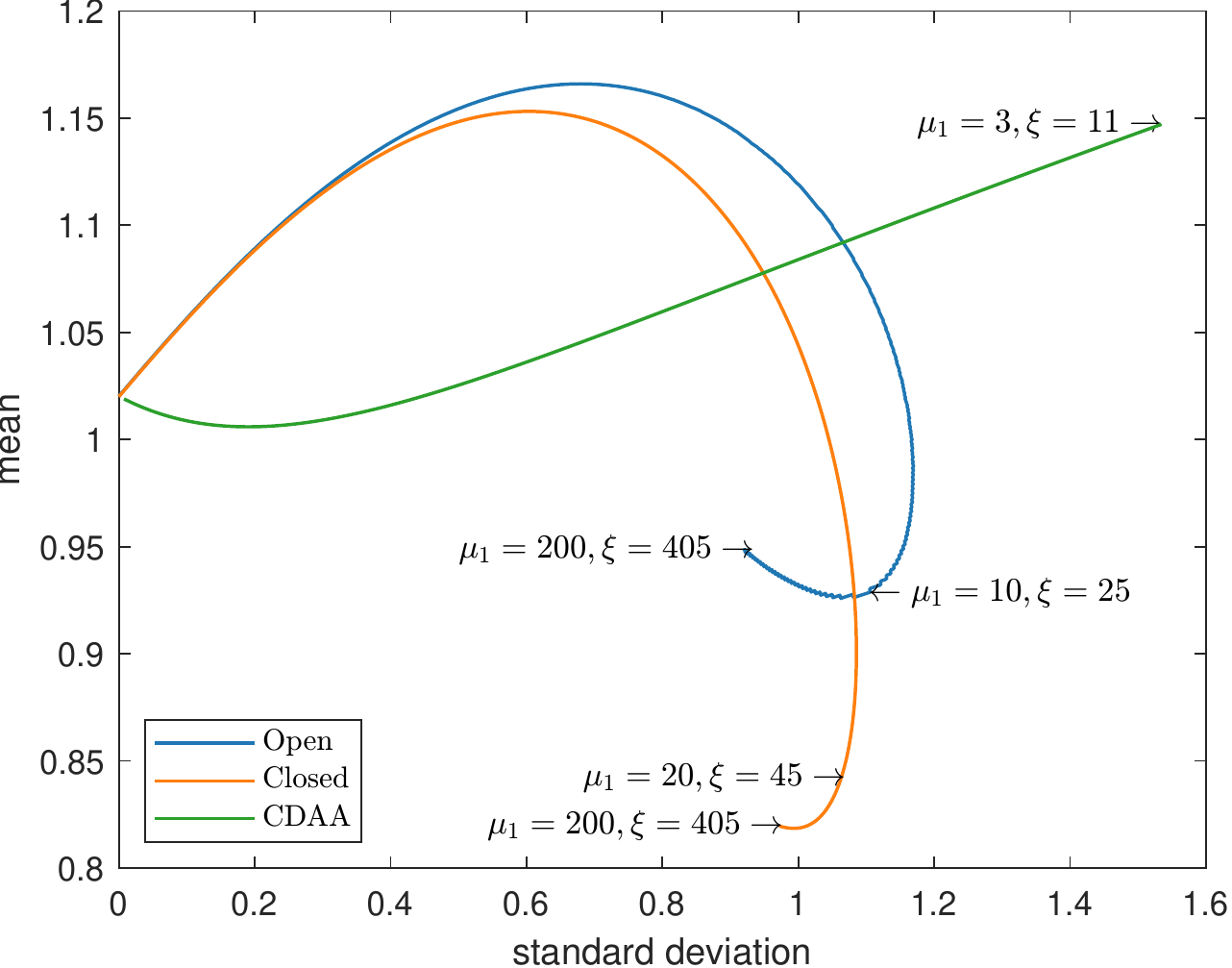}
			\caption{A linear relation between risk and ambiguity aversion}
			\label{fig:linear}
	\end{figure}
		
		Both the open-loop equilibrium in \eqref{eq:alpha*} and the closed-loop equilibrium in \cite{hpw21FS} share the same terminal value:
		\begin{equation}\label{}
			\alpha^*_T = \frac{\mu_1 B_T}{1 + \mu^2_1/\xi}.
		\end{equation}
		To ensure that $\alpha^*_T$ tends to infinity, a natural choice is to consider a quadratic relationship, $\xi = a_1 \mu^2_1$, with some constant $a_1 > 0$. Figure \ref{fig:quad} shows that, in contrast to Figure \ref{fig:linear}, the mean values and standard deviations do not converge, suggesting that the quadratic relation is a more plausible choice when $\mu_1$ is large.	
		
		\begin{figure}
			\centering
			\includegraphics[width=0.4\linewidth]{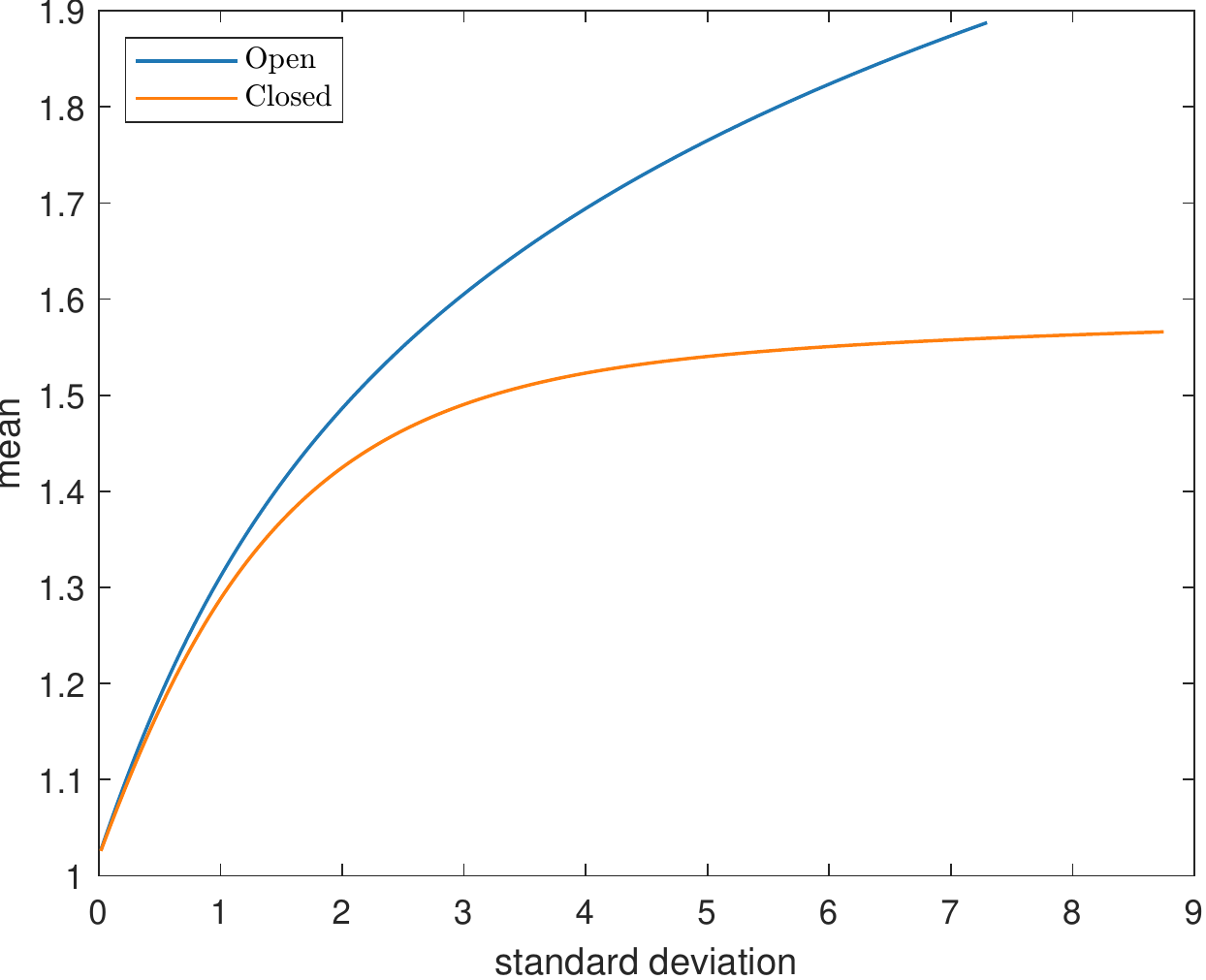}
			\caption{A quadratic relation between risk and ambiguity aversion}
			\label{fig:quad}
		\end{figure}

\section{Conclusions} \label{sec:con}
	With a general LQ framework and two semi-analytic examples, we unveil a complex interplay between risk and ambiguity. An open question is, does an agent with high risk aversion also have high ambiguity aversion? There are different pieces of evidence in the literature. A psychological study by \cite{borghans2009gender} documented that risk and ambiguity are related to cognitive and noncognitive traits with gender differences. Women are more risk averse than men, but women initially respond to ambiguity much more favorably than men, i.e., their reservation price as defined in \cite{borghans2009gender} does not decline. In contrast, to explain the equity premium puzzle, \cite[Table 2]{m04} showed two sets of aversion parameters for equilibrium asset pricing models to fit the empirical data in different periods. The high risk aversion case also has a higher ambiguity aversion parameter. Given these inconsistent results in the literature, a further study is needed and our analysis has provided a novel perspective from the equilibrium portfolio selection.

\section*{Acknowledgements}
This work began when Bingyan Han was a postdoctoral researcher in the Department of Mathematics at the University of Michigan. He expresses gratitude to the University of Michigan for providing support and an atmosphere conducive to this work. Bingyan Han is also partially supported by the HKUST (GZ) Start-up Fund G0101000197 and the Guangzhou-HKUST(GZ) Joint Funding Program (No. 2024A03J0630). Chi Seng Pun acknowledges the support from Singapore Ministry of Education Academic Research Fund Tier 2 [MOE-T2EP20220-0013]. Hoi Ying Wong acknowledges the support from the Research Grants Council of Hong Kong via GRF14308422.

\appendix
\section{Proofs}\label{sec:app}
\begin{proof}[Proof of Theorem \ref{Thm:hvariate}]
	Consider $t\in [0,T)$, $\varepsilon>0$, $\eta \in L^\infty_{\cF_t}(\Omega; \, \R^d)$, and $h^{t,\varepsilon,\eta}$ defined in \eqref{hspike}. Let $\zeta^{t,\varepsilon,\eta}(s;t)$ be the state process corresponding to $u^*$ and $h^{t,\varepsilon,\eta}$. Define $Y^\zeta$ and $Z^\zeta$ by
	\begin{align*}
		\left\{\begin{array}{lclll}
			d Y^\zeta_s  &=& \big(Y^\zeta_s h^*_s + \zeta^*(s;t) \eta \id_{s\in [t, t+\varepsilon)} \big)' dW^\p_s, & s\in[t,T], & Y^\zeta_t = 0,\\
			dZ^\zeta_s  &=& \big( Z^\zeta_s h^*_s + Y^\zeta_s \eta \id_{s\in [t, t+\varepsilon)} \big)' dW^\p_s, & s\in[t,T], & Z^\zeta_t = 0.
		\end{array}\right.
	\end{align*}
	Denote $R^\zeta_s := \zeta^{t,\varepsilon,\eta}(s;t) - \zeta^*(s;t) - Y^\zeta_s - Z^\zeta_s$. Note that $h^*$ is essentially bounded, one has the following standard estimates:
	\begin{align*}
		&\E^\p_t\Big[\sup_{s\in [t,T]} | Y^\zeta_s |^2\Big] = O(\varepsilon), \quad \E^\p_t\Big[\sup_{s\in [t,T]} |Z^\zeta_s |^2 \Big] = O(\varepsilon^2), \quad \E^\p_t\Big[ \sup_{s\in [t,T]} | R^\zeta_s |^2 \Big] = o(\varepsilon^2).
	\end{align*}
	Direct calculation leads to
	\begin{align*}
		&J(t, X^*_t, \zeta^*; u^*, h^{t,\varepsilon,\eta}) - J(t, X^*_t, \zeta^*; u^*, h^*) \\
		& \qquad =\E^\p_t \Big[ \int_t^T\Big( (Y^\zeta_s + Z^\zeta_s) f(X^*_s, u^*_s, h^*_s, s) \\
		& \hspace{2.7cm} -\frac{1}{2} \zeta^*(s;t) \ang{ (2 h^*_s + \eta)' \Phi(X^*_s, u^*_s, s), \eta} \id_{s\in [t, t+\varepsilon)}  \Big) ds \Big] \\
		&\qquad \quad + \E^\p_t \Big[ \big( \frac{1}{2} \ang{G X^*_T, X^*_T} - \ang{\nu \mathbb{E}^\p_t [\zeta^*(T;t) X^*_T], X^*_T} - \ang{ \mu_1 x_t+\mu_2, X^*_T} \big) \big(Y^\zeta_T + Z^\zeta_T\big) \Big]\\
		& \qquad \quad - \frac{1}{2} \ang{\nu\E^\p_t[Y^\zeta_T X^*_T], \E^\p_t[ Y^\zeta_T X^*_T]} + o(\varepsilon) \\
		& \qquad \quad + \E^\p_t \Big[ \int_t^T R^\zeta_s f(X^*_s, u^*_s, h^*_s, s) ds \Big] + \E^\p_t \Big[ \big( \frac{1}{2} \ang{G X^*_T, X^*_T} - \ang{ \mu_1 x_t+\mu_2, X^*_T} \big)  R^\zeta_T \Big] .
	\end{align*}
	Assumptions \ref{assum:state} and \ref{assum:obj} entail that $X^* \in L^4_\F(\Omega; \, C(0, \, T; \, \R^n))$ and $f(X^*, u^*, h^*, \cdot) \in L^2_\F(0,T; \R)$. With H\"{o}lder's inequality, the last two terms are of order $o(\varepsilon)$.
	
	By It\^o's lemma and the definition of $(p^\zeta(\cdot; t),k^\zeta(\cdot; t))$ in \eqref{pzeta}, the second term is reduced to
	\begin{align*}
		&\E^\p_t \Big[ \big( \frac{1}{2} \ang{G X^*_T, X^*_T} - \ang{\nu \mathbb{E}^\p_t [\zeta^*(T;t) X^*_T], X^*_T} - \ang{ \mu_1 x_t+\mu_2, X^*_T} \big) \big(Y^\zeta_T + Z^\zeta_T\big) \Big]\\
		&\qquad = - \E^\p_t \left[ \int_t^T \big( (h^*_s)' k^\zeta(s;t) + f(X^*_s, u^*_s, h^*_s, s) \big) (Y^\zeta_s + Z^\zeta_s) ds \right]\\
		&\qquad \quad + \E^\p_t \left[ \int_t^T \big( (Y^\zeta_s + Z^\zeta_s)h^*_s + \zeta^*(s;t) \eta \id_{s\in [t, t+\varepsilon)} + Y^\zeta_s \eta \id_{s\in [t, t+\varepsilon)} \big)' k^\zeta(s;t)ds \right] \\
		&\qquad = - \E^\p_t \Big[ \int_t^T  (Y^\zeta_s + Z^\zeta_s) f(X^*_s, u^*_s, h^*_s, s) ds \Big] \\
		& \qquad \quad + \E^\p_t \Big[ \int_t^T \zeta^*(s;t) \eta' k^\zeta(s;t)\id_{s\in [t, t+\varepsilon)}ds \Big] + o(\varepsilon).
	\end{align*}
	It follows that
	\begin{align*} 
		&J(t, X^*_t, \zeta^*; u^*, h^{t,\varepsilon,\eta}) - J(t, X^*_t, \zeta^*; u^*, h^*) \\
		& \qquad = \E^\p_t \left[ \int_t^{t+\varepsilon} \ang{\Lambda^\zeta(s;t),\eta} ds  \right] -\frac{1}{2} \E^\p_t \left[ \int_t^{t+\varepsilon} \zeta^*(s;t) \ang{\Phi(X^*_s, u^*_s, s)\eta, \eta} ds \right] \\
		& \qquad \quad - \frac{1}{2} \ang{\nu\E^\p_t[ Y^\zeta_T X^*_T], \E^\p_t[Y^\zeta_T X^*_T]} + o(\varepsilon).
	\end{align*}
	$\ang{\nu\E^\p_t[ Y^\zeta_T X^*_T], \E^\p_t[Y^\zeta_T X^*_T]}$ is of order $O(\varepsilon)$ in general. However, since $\nu \in \mathbb{S}^n$ and $\nu \succeq 0$, the term is non-negative. Since $\Phi(X^*_s, u^*_s, s) \succeq 0$, we obtain the claim as desired.
\end{proof}


\begin{proof}[Proof of Theorem \ref{Thm:uvariate}]
	The main idea is to consider Taylor expansions of the cost functionals up to the second order in spike variation.
	
	For any $t\in [0,T)$, $\varepsilon>0$ and $v\in L^4_{\cF_t}(\Omega; \, \R^l)$,  define $u^{t,\varepsilon,v}$ by \eqref{uspike}. Let $X^{t,\varepsilon,v}$ be the state process under $u^{t,\varepsilon,v}$ and $h^*$. Let $Y^x$ and $Z^x $ be the solutions to
	\begin{align*}
		\left\{\begin{array}{lcl}
			dY^x_s  &=&  A_s Y^x_s ds + \sum^d_{j=1} \big( C^j_s Y^x_s + D^j_s v \id_{s\in [t, t+\varepsilon)} \big) dW^\p_{js},  \;\; s\in[t,T], \\
			Y^x_t  &=& 0, \\
			dZ^x_s & = & \big(A_s Z^x_s + B' v \id_{s\in [t, t+\varepsilon)} \big)ds + \sum^d_{j=1} C^j_s Z^x_s dW^\p_{js},  \; s\in[t,T],\\
			Z^x_t & = & 0.
		\end{array}\right.
	\end{align*}
	It is immediate that the linear state process satisfies $X^{t,\varepsilon, v}_s = X^*_s + Y^x_s + Z^x_s$. Moreover, $\E^\p_t\Big[\sup_{s\in [t,T]} | Y^x_s |^2\Big] = O(\varepsilon)$ and $\E^\p_t\Big[\sup_{s\in [t,T]} |Z^x_s |^2 \Big] = O(\varepsilon^2)$.
	
	To expand $J(t, X^*_t, \zeta^*; u^{t,\varepsilon,v}, h^*) - J(t, X^*_t, \zeta^*; u^*, h^*)$, we consider the difference separately. The integral part leads to
	\begin{align*}
		& \E^\p_t \Big[\int_t^T\zeta^*(s;t) \frac{1}{2} \Big(\ang{Q_s X^{t,\varepsilon, v}_s, X^{t,\varepsilon, v}_s}+\ang{ R_s u^{t,\varepsilon, v}_s, u^{t,\varepsilon, v}_s}  \\
		& \hspace{3cm} - (h^*_s)' \Phi (X^{t,\varepsilon, v}_s, u^{t,\varepsilon, v}_s, s) h^*_s \Big) ds \Big] \\
		& - \E^\p_t \Big[\int_t^T\zeta^*(s;t) \frac{1}{2} \Big(\ang{Q_s X^*_s, X^*_s} + \ang{R_s u^*_s, u^*_s} - (h^*_s)' \Phi (X^*_s, u^*_s, s) h^*_s \Big) ds \Big] \\
		& = \E^\p_t \Big[\int_t^T\zeta^*(s;t) \Big\{ \frac{1}{2} \ang{Q_s (Y^x_s + Z^x_s), (Y^x_s + Z^x_s)} + \ang{Q_s X^*_s, (Y^x_s + Z^x_s)} \\ 
		& \hspace{3cm} + \frac{1}{2} \ang{R_s v, v} \id_{s\in [t, t+\varepsilon)} + \ang{R_s u^*_s, v} \id_{s\in [t, t+\varepsilon)} \\
		& \hspace{3cm} - \frac{1}{2} \sum^d_{j=1} \ang{\Xi_{1, j} (Y^x_s + Z^x_s), Y^x_s + Z^x_s} (h^*_{js})^2 \\
		&\hspace{3cm}  - \frac{1}{2} \sum^d_{j=1} \ang{\Xi_{2, j} v, v} (h^*_{js})^2 \id_{s\in [t, t+\varepsilon)} - \sum^d_{j=1} \ang{\Xi_{1, j} X^*_s, Y^x_s + Z^x_s} (h^*_{js})^2 \\
		& \hspace{3cm}  - \sum^d_{j=1} \ang{\Xi_{2, j} u^*_s, v} (h^*_{js})^2 \id_{s\in [t, t+\varepsilon)} \Big\} ds \Big].
	\end{align*}
	
	The terminal part yields
	{\allowdisplaybreaks
		\begin{align*}
			& \frac{1}{2} \E^\p_t [\zeta^*(T;t) \ang{G X^{t,\varepsilon, v}_T, X^{t,\varepsilon, v}_T}] - \ang{ \mu_1 x_t+\mu_2, \E^\p_t [\zeta^*(T;t) X^{t,\varepsilon, v}_T]} \\
			& \quad - \frac{1}{2} \E^\p_t [\zeta^*(T;t) \ang{G X^*_T, X^*_T}] + \ang{ \mu_1 x_t+\mu_2, \E^\p_t [\zeta^*(T;t) X^*_T]}\\
			& \quad - \frac{1}{2} \ang{\nu \E^\p_t [\zeta^*(T;t) X^{t,\varepsilon, v}_T], \E^\p_t [\zeta^*(T;t) X^{t,\varepsilon, v}_T]} \\ 
			& \quad + \frac{1}{2} \ang{\nu \E^\p_t [\zeta^*(T;t) X^*_T], \E^\p_t [\zeta^*(T;t) X^*_T]} \\
			& =  \frac{1}{2} \E^\p_t [\zeta^*(T;t) \ang{(G - \nu) (Y^x_T + Z^x_T), (Y^x_T + Z^x_T)}] + \E^\p_t [\zeta^*(T;t) \ang{G X^*_T, (Y^x_T + Z^x_T)}] \\
			& \quad - \ang{ \mu_1 x_t+\mu_2, \E^\p_t [\zeta^*(T;t) (Y^x_T + Z^x_T)]} - \ang{\nu \E^\p_t [\zeta^*(T;t) X^*_T], \E^\p_t [\zeta^*(T;t) (Y^x_T + Z^x_T)]} \\
			& \quad + \frac{1}{2} \E^\p_t [\zeta^*(T;t) \ang{\nu (Y^x_T + Z^x_T), (Y^x_T + Z^x_T)}] \\
			& \quad - \frac{1}{2} \ang{\nu \E^\p_t [\zeta^*(T;t) (Y^x_T + Z^x_T)], \E^\p_t [\zeta^*(T;t) (Y^x_T + Z^x_T)]}.
		\end{align*}
	}
	Since $ \nu \in \mathbb{S}^n$ and $ \nu \succeq 0$, the last two terms above satisfy
	\begin{align}
		& \frac{1}{2} \E^\p_t [\zeta^*(T;t) \ang{\nu (Y^x_T + Z^x_T), (Y^x_T + Z^x_T)}] \nonumber \\
		& \qquad - \frac{1}{2} \ang{\nu \E^\p_t [\zeta^*(T;t) (Y^x_T + Z^x_T)], \E^\p_t [\zeta^*(T;t) (Y^x_T + Z^x_T)]} \geq 0. \label{eq:YZineq}
	\end{align}
	Moreover, they are of order $O(\varepsilon)$ in general.
	
	By the definition of $Y^x$, $Z^x$, and $(p^x(\cdot;t), k^x(\cdot;t))$, an application of It\^o's lemma simplifies $\E^\p_t[ \zeta^*(T;t) \ang{ G X^*_T - \nu  \E^\p_t [\zeta^*(T;t) X^*_T] - \mu_1 x_t - \mu_2, (Y^x_T + Z^x_T)}]$ as
	\begin{align*}
		& \E^\p_t \Big[ \int^T_t \Big( - \zeta^*(s;t) \ang{f^*_x (X^*_s, u^*_s, s), (Y^x_s + Z^x_s)} \\
		& \hspace{1.5cm} + \ang{B_s p^x(s;t) + \sum^d_{j=1} (D^j_s)' k^{x, j}(s; t), v} \id_{s\in [t, t+\varepsilon)} \Big)ds \Big],
	\end{align*}
	and $ \E^\p_t [\zeta^*(T;t) \ang{(G - \nu) (Y^x_T + Z^x_T), (Y^x_T + Z^x_T)}]$ as
	{\allowdisplaybreaks
		\begin{align*}
			& \E^\p_t \Big[\int_t^T \Big\{ - \zeta^*(s;t) \ang{Q_s (Y^x_s + Z^x_s), (Y^x_s + Z^x_s)} \\
			& \hspace{1.5cm} + \zeta^*(s;t) \sum^d_{j=1} \ang{\Xi_{1, j} (Y^x_s + Z^x_s), Y^x_s + Z^x_s} (h^*_{js})^2 \\
			& \hspace{1.5cm} + 2 \ang{P(s;t) (Y^x_s + Z^x_s), B'_s v} \id_{s\in [t, t+\varepsilon)} \\
			& \hspace{1.5cm} + 2 \sum^d_{j=1} \ang{ K^j(s;t) (Y^x_s + Z^x_s), D^j_s v } \id_{s\in [t, t+\varepsilon)} \\
			& \hspace{1.5cm} + \sum^d_{j=1} \ang{ P(s;t) C^j_s (Y^x_s + Z^x_s), D^j_s v} \id_{s\in [t, t+\varepsilon)}  \\
			& \hspace{1.5cm} + \sum^d_{j=1} \ang{(D^j_s)' P(s;t) D^j_s v, v} \id_{s\in [t, t+\varepsilon)} \Big\} ds \Big] \\
			& =  \E^\p_t \Big[\int_t^T \Big\{ - \zeta^*(s;t) \ang{Q_s (Y^x_s + Z^x_s), (Y^x_s + Z^x_s)} \\
			& \hspace{1.5cm} + \zeta^*(s;t) \sum^d_{j=1} \ang{\Xi_{1, j} (Y^x_s + Z^x_s), Y^x_s + Z^x_s} (h^*_{js})^2 \\
			& \hspace{1.5cm} + \sum^d_{j=1} \ang{(D^j_s)' P(s;t) D^j_s v, v} \id_{s\in [t, t+\varepsilon)} \Big\} ds \Big] + o(\varepsilon).
		\end{align*}
	}
	Summing up and we obtain
	\begin{align}
		& J(t, X^*_t, \zeta^*; u^{t,\varepsilon,v}, h^*) - J(t, X^*_t, \zeta^*; u^*, h^*) \\
		& \qquad \geq \E^\p_t \Big[ \int_t^{t+\varepsilon} \ang{\Lambda^x(s; t), v} ds \Big] \nonumber + \frac{1}{2} \E^\p_t \Big[ \int_t^{t+\varepsilon} \ang{\Sigma(s;t) v ,v} ds \Big] + o(\varepsilon). \nonumber
	\end{align}
	The inequality is from the property of $Y^x$ and $Z^x$ in \eqref{eq:YZineq}. The inequality is also tight. Therefore, a sufficient condition is given in \eqref{eq:u-suff}.
\end{proof}


\begin{proof}[Proof of Corollary \ref{cor:special}]
	Denote the Brownian motion and the conditional expectation corresponding to $h^*$ by $W^*$ and $\E^*_t[\cdot]$, respectively. Then
	\begin{align*}
		\left\{\begin{array}{lcl}
			dY^x_s  &=&  \big(A_s Y^x_s + \sum^d_{j=1} h^*_{js} D^j_s v \id_{s\in [t, t+\varepsilon)} \big) ds \\
			& & + \sum^d_{j=1} D^j_s v \id_{s\in [t, t+\varepsilon)} dW^*_{js},  \;\; s\in[t,T], \\
			Y^x_t  &=& 0.
		\end{array}\right.
	\end{align*}
	Since $A$ is deterministic, it is direct to prove that $\E^\p_t [\zeta^*(s;t) Y^x_s] = \E^*_t [Y^x_s] = O(\varepsilon)$ for $s \in [t, T]$. Therefore, $\ang{\nu \E^\p_t [\zeta^*(T;t) Y^x_T], \E^\p_t [\zeta^*(T;t) Y^x_T]} = o(\varepsilon)$. We do not need to subtract $G$ by $\nu$ to bound it. \eqref{eq:YZineq} is not needed. Then the claim follows in the same manner as in Theorem \ref{Thm:uvariate}.
\end{proof}


\begin{proof}[Proof of Lemma \ref{Lem:StateODE}]
	For simplicity, we omit the time script $s$ when no confusion arises. With $h^*$ in \eqref{eq:h*}, we first simplify the ODE system \eqref{eq:LHF}-\eqref{eq:MNGamma} to 
	\begin{align}\label{eq:Simp-StateODE}
		\left\{\begin{array}{lcl}
			\dot{L} &=& - [2(A + B' \alpha^*) + (2 E/\xi_1 + 1)| C + D \alpha^*|^2] L - Q/2  \\
			& & - \ang{R \alpha^*, \alpha^*}/2 + E^2 |C + D \alpha^*|^2/(2\xi_1), \\
			\dot{H} &=& - [2(A + B' \alpha^*) + 2|C + D \alpha^*|^2 E/\xi_1] H, \\
			\dot{F} &=& - [(A + B' \alpha^*) + |C + D \alpha^*|^2 E/\xi_1] F, \\
			\dot{M} &=& - [2A + B' \alpha^* + (2C + D \alpha^*)'(C + D \alpha^*) E/\xi_1 + C'(C + D \alpha^*)] M  \\
			& & - Q + E^2 |C + D \alpha^*|^2/\xi_1, \\
			\dot{N} &=& - [2A + B' \alpha^* + (2C + D \alpha^*)'(C + D \alpha^*) E/\xi_1] N, \\
			\dot{\Gamma} &=& - [A + C'(C + D \alpha^*) E/\xi_1] \Gamma.
		\end{array}\right.
	\end{align}
	Truncate $\Delta$, $E$, $M$, and $L$ as follows. With given constants $\ubar{m}$, $\bar{m}$, $\ubar{\delta}$, and $\ubar{e}$, consider functions 
	\begin{align*}
		\Delta_c &= \max\{\ubar{\delta}, \min\{\Delta, 0\}\}, \quad E_c = \max\{\ubar{e}, \min\{E, 0\}\}, \\
		M_c &= \max\{\ubar{m}, \min\{M, \bar{m}\}\}, \quad L_c := \max\{L, 0\},
	\end{align*}
	such that $\ubar{\delta} \leq \Delta_c \leq 0$, $\ubar{e} \leq E_c \leq 0$, $0 < \ubar{m} \leq M_c \leq \bar{m}$, and $L_c \geq 0$. Replace the ODE system \eqref{eq:Simp-StateODE} with truncated terms on the right hand side of \eqref{eq:Simp-StateODE}. Denote $\alpha^*_c$ as $\alpha^*$ with $\Delta, E, M$ replaced by $\Delta_c, E_c$, and $M_c$.  
	
	First, we derive bounds on $B'\alpha^*_c$, $C'D \alpha^*_c$, $(\alpha^*_c)' D'D \alpha^*_c$, and $(\alpha^*_c)' R \alpha^*_c$. Denote $W := R + (M_c + \Delta_c E_c/\xi_1) D'D$. Note that $ \Delta_c E_c \geq 0$, then
	\begin{equation*}
		\tau I \preceq  R + \ubar{m} D'D \preceq W \preceq R + (\bar{m} + \ubar{\delta}\ubar{e} /\xi_1) D'D.
	\end{equation*} 
	To make notations concrete, we denote this inequality as $ \ubar{W} \preceq W \preceq \bar{W}$. Observe that
	\begin{equation*}
		B'\alpha^*_c = - B' W^{-1} D' C (M_c + \Delta_c E_c/\xi_1)  - B' W^{-1} B \Delta_c,
	\end{equation*}
	and $-(B'W^{-1} B + C' D W^{-1} D'C)/2 \leq B' W^{-1} D' C \leq (B'W^{-1} B + C' D W^{-1} D'C)/2$, then for any $ s \in [0, T]$,
	\begin{align*}
		B'\alpha^*_c & \geq - (\bar{m} + \ubar{\delta} \ubar{e}/ \xi_1)(B'\ubar{W}^{-1} B + C' D \ubar{W}^{-1} D'C)/2 =: \ubar{b}, \\
		B'\alpha^*_c & \leq (\bar{m} + \ubar{\delta} \ubar{e}/ \xi_1)(B'\ubar{W}^{-1} B + C' D \ubar{W}^{-1} D'C)/2  - \ubar{\delta} B' \ubar{W}^{-1} B =: \bar{b}.
	\end{align*}
	In a similar manner, bounds for other terms are given by
	\begin{align*}
		C'D\alpha^*_c & \geq - (\bar{m} + \ubar{\delta} \ubar{e}/ \xi_1)C' D \ubar{W}^{-1} D'C + \ubar{\delta}(B'\ubar{W}^{-1} B + C' D \ubar{W}^{-1} D'C)/2 =: \ubar{c}, \\
		C'D \alpha^*_c & \leq - \ubar{m} C'D \bar{W}^{-1} D'C - \ubar{\delta}(B'\ubar{W}^{-1} B + C' D \ubar{W}^{-1} D'C)/2 =: \bar{c}.
	\end{align*}
	$(\alpha^*_c)' D'D \alpha^*_c \geq 0$ and
	\begin{align*}
		(\alpha^*_c)' D'D \alpha^*_c  \leq 2 \ubar{\delta}^2 B' \ubar{W}^{-1} D'D \ubar{W}^{-1} B + 2 (\bar{m} + \ubar{\delta} \ubar{e}/ \xi_1)^2 C'D \ubar{W}^{-1} D'D \ubar{W}^{-1} D'C =: \bar{d}.
	\end{align*}
	$(\alpha^*_c)' R \alpha^*_c \geq 0$ and
	\begin{align*}
		(\alpha^*_c)' R \alpha^*_c  \leq 2 \ubar{\delta}^2 B' \ubar{W}^{-1} R \ubar{W}^{-1} B + 2 (\bar{m} + \ubar{\delta} \ubar{e}/ \xi_1)^2 C'D \ubar{W}^{-1} R \ubar{W}^{-1} D'C =: \bar{r}.
	\end{align*}
	Crucially, these bounds $\bar{b}$, $\ubar{b}$, $\bar{c}$, $\ubar{c}$, $\bar{d}$, and $\bar{r}$ do not rely on solutions to the truncated system. It follows that the derivatives are bounded by
	\begin{align*}
		\dot{L}_s &\leq - [2 A_s + 2 \ubar{b}_s + 2 \ubar{e}_s (|C_s|^2 + 2 \bar{c}_s + \bar{d}_s)/\xi_1] L_s - Q_s/2 + \ubar{e}^2_s  (|C_s|^2 + 2 \bar{c}_s + \bar{d}_s)/(2\xi_1) \\
		&=: - U_{L,1}(s) L_s - U_{L, 0}(s), \\
		\dot{L}_s &\geq - [2 A_s + 2 \bar{b}_s + (|C_s|^2 + 2 \bar{c}_s + \bar{d}_s)] L_s - Q_s/2 - \bar{r}_s/2 =: - V_{L,1}(s) L_s - V_{L, 0}(s), \\
		\dot{H}_s &\leq - [2 A_s + 2 \ubar{b}_s + 2 \ubar{e}_s (|C_s|^2 + 2 \bar{c}_s + \bar{d}_s)/\xi_1] H_s =: - U_{H,1}(s) H_s, \\
		\dot{H}_s &\geq - [2 A_s + 2 \bar{b}_s] H_s =: - V_{H,1}(s) H_s, \\
		\dot{F}_s &\leq - [A_s + \ubar{b}_s + \ubar{e}_s (|C_s|^2 + 2 \bar{c}_s + \bar{d}_s)/\xi_1] F_s =: - U_{F,1}(s) F_s, \\
		\dot{F}_s &\geq - [A_s + \bar{b}_s] F_s =: - V_{F,1}(s) F_s,
	\end{align*}
	and	
	\begin{align*}
		\dot{M}_s \leq& - [2 A_s + \ubar{b}_s + \ubar{e}_s \max\{ 2|C_s|^2 + 3\bar{c}_s + \bar{d}_s, 0 \} /\xi_1 + |C_s|^2 + \ubar{c}_s ] M_s \\
		& - Q_s + \ubar{e}^2_s  (|C_s|^2 + 2 \bar{c}_s + \bar{d}_s)/\xi_1 \\
		=:& - U_{M,1}(s) M_s - U_{M, 0}(s), \\
		\dot{M}_s \geq& - [2 A_s + \bar{b}_s + \ubar{e}_s \min\{ 2|C_s|^2 + 3\ubar{c}_s, 0\} /\xi_1 + |C_s|^2 + \bar{c}_s ] M_s - Q_s \\
		=:& - V_{M,1}(s) M_s - V_{M, 0}(s), \\
		\dot{N}_s \leq& - [2 A_s + \ubar{b}_s + \ubar{e}_s \max\{ 2|C_s|^2 + 3\bar{c}_s + \bar{d}_s, 0 \} /\xi_1] N_s =: - U_{N,1}(s) N_s, \\
		\dot{N}_s \geq& - [2 A_s + \bar{b}_s + \ubar{e}_s \min\{ 2|C_s|^2 + 3\ubar{c}_s, 0\} /\xi_1 ] N_s =: - V_{N,1}(s) N_s, \\
		\dot{\Gamma}_s \leq& - [A_s + \ubar{e}_s \max\{ |C_s|^2 + \bar{c}_s, 0 \} /\xi_1] \Gamma_s =: - U_{\Gamma,1}(s) \Gamma_s, \\
		\dot{\Gamma}_s \geq& - [A_s + \ubar{e}_s \min\{|C_s|^2 + \ubar{c}_s, 0\} /\xi_1 ]  \Gamma_s =: - V_{\Gamma,1}(s) \Gamma_s.
	\end{align*}
	Thus, the truncated system is locally Lipschitz with linear growth. The existence and uniqueness of the solution to the truncated system is then guaranteed. By comparison theorem of ODEs, we obtain that
	\begin{align*}
		L_c(t) \leq& \frac{G}{2} e^{\int^T_t V_{L,1}(s) ds} + \int^T_t  e^{\int^u_t V_{L,1}(s) ds} V_{L, 0}(u) du =: \bar{L}(t), \\
		L_c(t) \geq& \frac{G}{2} e^{\int^T_t U_{L,1}(s) ds} + \int^T_t  e^{\int^u_t U_{L,1}(s) ds} U_{L, 0}(u) du =: \ubar{L}(t).
	\end{align*}
	Similarly, we can define $\ubar{M}(t)$, $\bar{M}(t)$, $\ubar{N}(t)$, $\bar{N}(t)$, $\ubar{\Gamma}(t)$, $\bar{\Gamma}(t)$, $\ubar{H}(t)$, $\bar{H}(t)$, $\ubar{F}(t)$, and $\bar{F}(t)$ with $U$ and $V$ that are independent of solutions to ODEs. Since \eqref{cond:StateODE} holds, then the truncation constants are not binding. The truncated solution is indeed the solution to the original ODE system.
\end{proof}

\begin{proof}[Proof of Proposition \ref{Cor:StateAdm}]
	We only need to verify the moment conditions in Assumption \ref{assum:state}. Lemma \ref{Lem:StateODE} implies the ODE system has a continuous solution pair and $M \geq \ubar{m} > 0$. Then $(h^*, \alpha^*)$ is essentially bounded. It is immediate that $u^* \in L^4_\F(0, \, T; \, \R^l)$. Since sufficient conditions \eqref{eq:h-suff} and \eqref{eq:u-suff} hold, $(h^*, u^*)$ is an equilibrium. 
\end{proof}


\begin{proof}[Proof of Lemma \ref{Lem:CtrlODE}]
	We truncate $\Delta$, $E$, and $M$ as $\Delta_c$, $E_c$, and $M_c$ satisfying
	$|\Delta_c| \leq \bar{\delta}$,  $|E_c| \leq \bar{e}$,  and $0< \ubar{m} \leq M_c \leq \bar{m}$.
	
	The discriminant of the quadratic equation \eqref{eq:quad-alpha} satisfies
	\begin{align*}
		\beta^2_j - 4 \kappa \gamma =& B^2_{jt} \Delta^2_t - 4 (D^2_t M_t + R_t) D^2_t(\Delta_t E_t - E^2_t)/(\xi_2 l) \\
		= & [4 (D^2_t M_t + R_t) D^2_t (E_t - \Delta_t/2)^2 + (\xi_2 l B^2_{jt} - D^2_t (D^2_t M_t + R_t)) \Delta^2_t]   /(\xi_2 l).
	\end{align*}
	Under the condition \eqref{eq:discriminant}, the discriminant is non-negative with the truncated solutions. Since $\alpha^*_j$ is a root of the quadratic equation \eqref{eq:quad-alpha}, it satisfies
	\begin{align*}
		\left(\alpha^*_j + \frac{\beta_j}{2 \kappa} \right)^2 &= \frac{\beta^2_j - 4 \kappa \gamma}{4 \kappa^2} \leq \frac{\beta^2_j}{4 \kappa^2} + \frac{D^2}{\xi_2 l (D^2 M_c + R)} (E^2_c + |\Delta_c E_c|).
	\end{align*}
	With the elementary inequality $\sqrt{a^2+b^2} \leq |a| + |b|$, we obtain
	\begin{align*}
		\left|\alpha^*_j + \frac{\beta_j}{2 \kappa} \right| & \leq \frac{|\beta_j|}{2 |\kappa|} + \frac{|D|}{\sqrt{\xi_2 l (D^2 M_c + R)}} \sqrt{E^2_c + |\Delta_c E_c|}.
	\end{align*}
	Therefore, with truncated solutions, it follows that
	\begin{align*}
		\left|\alpha^*_j  \right| & \leq \frac{|\beta_j|}{|\kappa|} + \frac{|D|}{\sqrt{\xi_2 l (D^2 M_c + R)}} \sqrt{E^2_c + |\Delta_c E_c|} \\
		& \leq \frac{|B_j| \bar{\delta}}{D^2 \ubar{m} + R} + \frac{|D|\sqrt{\bar{e}^2 + \bar{\delta} \bar{e}}}{\sqrt{\xi_2 l (D^2 \ubar{m} + R)}} =: \bar{\alpha}_j.
	\end{align*}
	We can bound the following terms as $|B'\alpha^*| \leq \sum^l_{j=1} |B_j| \bar{\alpha}_j =: \bar{b}$ and $0 \leq |\alpha^*|^2 \leq \sum^l_{j=1} \bar{\alpha}^2_j =: |\bar{\alpha}|^2$.
	The derivatives with truncated solutions are bounded as
	\begin{align*}
		\dot{L}_s \leq& -[2A_s - 2 \bar{b}_s - 2 D^2_s \bar{e}_s/\xi_2] L_s - Q_s/2 + \bar{e}^2_s D^2_s /(2\xi_2) \\
		=:& - U_{L,1}(s) L_s - U_{L, 0}(s), \\
		\dot{L}_s \geq& -[2A_s + 2 \bar{b}_s + 2 D^2_s \bar{e}_s/\xi_2 + D^2_s |\bar{\alpha}_s|^2] L_s - Q_s/2 - R_s |\bar{\alpha}_s|^2/2 \\
		=:& - V_{L,1}(s) L_s - V_{L, 0}(s), \\
		\dot{H}_s \leq& -[2A_s - 2\bar{b}_s - 2 D^2_s \bar{e}_s/\xi_2 ] H_s =: - U_{H,1}(s) H_s, \\
		\dot{H}_s \geq& -[2A_s + 2\bar{b}_s + 2 D^2_s \bar{e}_s/\xi_2] H_s =: - V_{H,1}(s) H_s, \\
		\dot{F}_s \leq& -[A_s - \bar{b}_s - D^2_s \bar{e}_s/\xi_2 ] F_s =: - U_{F,1}(s) F_s, \\
		\dot{F}_s \geq& -[A_s + \bar{b}_s + D^2_s \bar{e}_s/\xi_2] F_s =: - V_{F,1}(s) F_s.
	\end{align*}
	Denote $I := M - N$, then
	\begin{align*}
		\dot{M}_s \leq& -[2A_s - \bar{b}_s - D^2_s \bar{e}_s/\xi_2 ] M_s - Q_s =: - U_{M,1}(s) M_s - Q_s, \\
		\dot{M}_s \geq& -[2A_s + \bar{b}_s + D^2_s \bar{e}_s/\xi_2] M_s - Q_s =: - V_{M,1}(s) M_s - Q_s,\\
		\dot{I}_s \leq& -[2A_s - \bar{b}_s - D^2_s \bar{e}_s/\xi_2 ] I_s - Q_s =: - U_{I,1}(s) I_s - Q_s,\\
		\dot{I}_s \geq& -[2A_s + \bar{b}_s + D^2_s \bar{e}_s/\xi_2] I_s - Q_s =: - V_{I,1}(s) I_s - Q_s.
	\end{align*}
	
	In the same spirit of Lemma \ref{Lem:StateODE}, existence and uniqueness of the solution to the truncated system are guaranteed since the system is locally Lipschitz with linear growth. We can define $\ubar{L}(t)$, $\bar{L}(t)$, $\ubar{M}(t)$, $\bar{M}(t)$, $\ubar{I}(t)$, $\bar{I}(t)$, $\ubar{H}(t)$, $\bar{H}(t)$, $\ubar{F}(t)$, and $\bar{F}(t)$ with $U$ and $V$ that are independent of solutions to ODEs. With \eqref{cond:CtrlODE}, the truncation constants are not binding. The truncated solution is indeed the solution to the original ODE system.	
\end{proof}


\begin{proof}[Proof of Proposition \ref{Cor:CtrlAdm}]
	We only need to verify the moment conditions in Assumption \ref{assum:state}. The ODE system has a continuous solution pair and $M \geq \ubar{m} > 0$. Moreover, with $|\alpha^*_{jt}| \geq \phi$, we conclude that $(h^*, \alpha^*)$ is essentially bounded. Then $u^* \in L^4_\F(0, \, T; \, \R^l)$. Since sufficient conditions \eqref{eq:h-suff} and \eqref{eq:u-suff} hold, $(h^*, u^*)$ is an equilibrium. 
\end{proof}


\begin{proof}[Proof of Corollary \ref{Cor:limit}]
	We only prove the open-loop case since the closed-loop argument is almost the same.
	
	If $\mu_1 = 0$, we verify $h^* = 0$ and $u^* = 0$ is an equilibrium pair. With this conjectured solution, the ODE system \eqref{eq:LHF}--\eqref{eq:MNGamma} reduces to
	\begin{align*}
		&\dot{L}_s = - 2 A_s L_s, \; L_T = \frac{G}{2}, \quad \dot{H}_s = - 2 A_s H_s, \; H_T = \nu, \quad \dot{F}_s = - A_s F_s, \; F_T = 0, \\
		& \dot{M}_s = - 2A_s  M_s, \; M_T = G, \quad  \dot{N}_s = - 2A_s  N_s, \; N_T = \nu, \quad \dot{\Gamma}_s = - A_s  \Gamma_s, \; \Gamma_T = 0.
	\end{align*}
	Then $F = \Gamma = 0$. $M = N$ and $2L = H$ since $G = \nu = 1$. Thus, $\Delta = 0$. By the expression of $h^*$ and $\alpha^*$ in \eqref{eq:h*}--\eqref{eq:alpha*}, we obtain $\alpha^* = 0$ and $h^* = 0$ as desired.
	
	If $\mu_1 \rightarrow \infty$, conjecture $h^* = -B$ and $u^* = 0$ is an equilibrium pair. The ODE system \eqref{eq:LHF} and \eqref{eq:MNGamma} reduces to
	\begin{align*}
		&\dot{L}_s = - 2 A_s L_s +  \xi_1 \frac{|h^*_s|^2}{2}, \; L_T = \frac{G}{2}, \quad \dot{H}_s = - 2 A_s H_s, \; H_T = \nu, \\
		& \dot{F}_s = - A_s F_s, \; F_T = \mu_1 \rightarrow \infty, \\
		& \dot{M}_s = - 2A_s  M_s +  \xi_1 |h^*_s|^2, \; M_T = G, \quad \dot{N}_s = - 2A_s  N_s, \; N_T = \nu, \\
		& \dot{\Gamma}_s = - A_s  \Gamma_s, \; \Gamma_T = \mu_1 \rightarrow \infty.
	\end{align*}
	Therefore, $F \sim \mu_1$, $\Gamma \sim \mu_1$ and converge to infinity. As $L$, $H$, $M$ and $N$ are still bounded, then $\Delta \sim \mu_1$, $E \sim \mu_1$ and converge to infinity. Thus,
	\begin{align*}
		\alpha^*_t &= - [M_t + \Delta_t E_t/\xi_1]^{-1}  B_t \Delta_t  \rightarrow 0, \\
		h^*_t &= E_t \alpha^*_t / \xi_1 = - [\xi_1 M_t + \Delta_t E_t]^{-1}  B_t \Delta_t E_t  \rightarrow - B_t.
	\end{align*}
\end{proof}

\begin{proof}[Proof of Corollary \ref{Cor:Ctrl}]
	It is similar to the proof of Corollary \ref{Cor:limit}. If $\mu_1 = 0$, we conjecture $h^* = 0$, $u^* = 0$, $2L = H$, $M=N$, $F=0$, $\Gamma=0$, $E=0$, and $\Delta=0$. Then the ODE system \eqref{eq:CtrlODE} reduces to
	\begin{eqnarray}\label{eq:Ctrl_simple}
		\left\{\begin{array}{lcl}
				\dot{L}_s &=& -2A_s L_s, \quad L_T = 1/2,\\
				\dot{H}_s &=& -2A_s H_s, \quad H_T = 1, \\
				\dot{F}_s &=& -A_s F_s, \quad F_T = 0, \\
				\dot{M}_s &=& -2A_s M_s, \quad M_T = 1, \\
				\dot{N}_s &=& -2A_s N_s, \quad N_T = 1, \\
				\dot{\Gamma}_s &=& - A_s \Gamma_s, \quad \Gamma_T = 0.
		\end{array}\right.
	\end{eqnarray}
	The solution to \eqref{eq:Ctrl_simple} agrees with our conjecture.
\end{proof}


\begin{thebibliography}{10}
	
	\bibitem{ahs03}
	Evan~W Anderson, Lars~Peter Hansen, and Thomas~J Sargent.
	\newblock A quartet of semigroups for model specification, robustness, prices
	of risk, and model detection.
	\newblock {\em Journal of the European Economic Association}, 1(1):68--123,
	2003.
	
	\bibitem{bc10}
	Suleyman Basak and Georgy Chabakauri.
	\newblock Dynamic mean-variance asset allocation.
	\newblock {\em The Review of Financial Studies}, 23(8):2970--3016, 2010.
	
	\bibitem{bayraktar2021equilibrium}
	Erhan Bayraktar, Jingjie Zhang, and Zhou Zhou.
	\newblock Equilibrium concepts for time-inconsistent stopping problems in
	continuous time.
	\newblock {\em Mathematical Finance}, 31(1):508--530, 2021.
	
	\bibitem{bp17}
	Sara Biagini and Mustafa~{\c{C}} P{\i}nar.
	\newblock The robust {Merton} problem of an ambiguity averse investor.
	\newblock {\em Mathematics and Financial Economics}, 11(1):1--24, 2017.
	
	\bibitem{bkm17}
	Tomas Bj{\"o}rk, Mariana Khapko, and Agatha Murgoci.
	\newblock On time-inconsistent stochastic control in continuous time.
	\newblock {\em Finance and Stochastics}, 21(2):331--360, 2017.
	
	\bibitem{bm14}
	Tomas Bj{\"o}rk and Agatha Murgoci.
	\newblock A theory of {Markovian} time-inconsistent stochastic control in
	discrete time.
	\newblock {\em Finance and Stochastics}, 18(3):545--592, 2014.
	
	\bibitem{bmz14}
	Tomas Bj{\"o}rk, Agatha Murgoci, and Xun~Yu Zhou.
	\newblock Mean--variance portfolio optimization with state-dependent risk
	aversion.
	\newblock {\em Mathematical Finance}, 24(1):1--24, 2014.
	
	\bibitem{borghans2009gender}
	Lex Borghans, James~J Heckman, Bart~HH Golsteyn, and Huub Meijers.
	\newblock Gender differences in risk aversion and ambiguity aversion.
	\newblock {\em Journal of the European Economic Association}, 7(2-3):649--658,
	2009.
	
	\bibitem{Cohen2020}
	Jonathan Cohen, Keith~Marzilli Ericson, David Laibson, and John~Myles White.
	\newblock Measuring time preferences.
	\newblock {\em Journal of Economic Literature}, 58(2):299--347, June 2020.
	
	\bibitem{Ekeland2006}
	Ivar Ekeland and Ali Lazrak.
	\newblock Being serious about non-commitment: Subgame perfect equilibrium in
	continuous time.
	\newblock {\em arXiv: math/0604264}, April 2006.
	
	\bibitem{e61}
	Daniel Ellsberg.
	\newblock Risk, ambiguity, and the {Savage} axioms.
	\newblock {\em The Quarterly Journal of Economics}, 75(4):643--669, 1961.
	
	\bibitem{fleming2006controlled}
	Wendell~H Fleming and Halil~Mete Soner.
	\newblock {\em Controlled {Markov} processes and viscosity solutions},
	volume~25.
	\newblock Springer Science \& Business Media, New York, NY, 2006.
	
	\bibitem{hpw21FS}
	Bingyan Han, Chi~Seng Pun, and Hoi~Ying Wong.
	\newblock Robust state-dependent mean--variance portfolio selection: a
	closed-loop approach.
	\newblock {\em Finance and Stochastics}, pages 1--33, 2021.
	
	\bibitem{han2022robust}
	Bingyan Han, Chi~Seng Pun, and Hoi~Ying Wong.
	\newblock Robust time-inconsistent stochastic linear-quadratic control with
	drift disturbance.
	\newblock {\em Applied Mathematics \& Optimization}, 86(1):1--40, 2022.
	
	\bibitem{He2022}
	Xue~Dong He and Xun~Yu Zhou.
	\newblock Who are {I}: Time inconsistency and intrapersonal conflict and
	reconciliation.
	\newblock In {\em Stochastic Analysis, Filtering, and Stochastic Optimization},
	pages 177--208. Springer International Publishing, Cham, 2022.
	
	\bibitem{hjz12}
	Ying Hu, Hanqing Jin, and Xun~Yu Zhou.
	\newblock Time-inconsistent stochastic linear--quadratic control.
	\newblock {\em SIAM Journal on Control and Optimization}, 50(3):1548--1572,
	2012.
	
	\bibitem{hjz17}
	Ying Hu, Hanqing Jin, and Xun~Yu Zhou.
	\newblock Time-inconsistent stochastic linear-quadratic control:
	characterization and uniqueness of equilibrium.
	\newblock {\em SIAM Journal on Control and Optimization}, 55(2):1261--1279,
	2017.
	
	\bibitem{hh17}
	Jianhui Huang and Minyi Huang.
	\newblock Robust mean field linear-quadratic-{Gaussian} games with unknown
	{$L^2$}-disturbance.
	\newblock {\em SIAM Journal on Control and Optimization}, 55(5):2811--2840,
	2017.
	
	\bibitem{huang2017char}
	Jianhui Huang, Xun Li, and Tianxiao Wang.
	\newblock Characterizations of closed-loop equilibrium solutions for dynamic
	mean--variance optimization problems.
	\newblock {\em Systems \& Control Letters}, 110:15--20, 2017.
	
	\bibitem{huang2021strong}
	Yu-Jui Huang and Zhou Zhou.
	\newblock Strong and weak equilibria for time-inconsistent stochastic control
	in continuous time.
	\newblock {\em Mathematics of Operations Research}, 46(2):428--451, 2021.
	
	\bibitem{ih18}
	Amine Ismail and Huy{\^e}n Pham.
	\newblock Robust {Markowitz} mean-variance portfolio selection under ambiguous
	covariance matrix.
	\newblock {\em Mathematical Finance}, 29(1):174--207, 2019.
	
	\bibitem{k21}
	Frank~H Knight.
	\newblock {\em Risk, uncertainty and profit}.
	\newblock Houghton Mifflin, New York, NY, 1921.
	
	\bibitem{Lei2020}
	Qian Lei and Chi~Seng Pun.
	\newblock An extended {McKean}--{Vlasov} dynamic programming approach to robust
	equilibrium controls under ambiguous covariance matrix.
	\newblock {\em {SSRN} Electronic Journal}, 2020.
	
	\bibitem{Lei2021}
	Qian Lei and Chi~Seng Pun.
	\newblock Nonlocal fully nonlinear parabolic differential equations arising in
	time-inconsistent problems.
	\newblock {\em arXiv: 2110.04237}, October 2021.
	
	\bibitem{li2023robust}
	Mengge Li, Shuaijie Qian, and Chao Zhou.
	\newblock Robust equilibrium strategy for mean-variance portfolio selection.
	\newblock {\em arXiv preprint arXiv:2305.07166}, 2023.
	
	\bibitem{lz02}
	Andrew~EB Lim and Xun~Yu Zhou.
	\newblock Mean-variance portfolio selection with random parameters in a
	complete market.
	\newblock {\em Mathematics of Operations Research}, 27(1):101--120, 2002.
	
	\bibitem{m04}
	Pascal~J Maenhout.
	\newblock Robust portfolio rules and asset pricing.
	\newblock {\em Review of Financial Studies}, 17(4):951--983, 2004.
	
	\bibitem{moon2020linear}
	Jun Moon and Hyun~Jong Yang.
	\newblock Linear-quadratic time-inconsistent mean-field type {Stackelberg}
	differential games: Time-consistent open-loop solutions.
	\newblock {\em IEEE Transactions on Automatic Control}, 66(1):375--382, 2020.
	
	\bibitem{os14}
	Bernt {\O}ksendal and Agn{\`e}s Sulem.
	\newblock Forward--backward stochastic differential games and stochastic
	control under model uncertainty.
	\newblock {\em Journal of Optimization Theory and Applications}, 161(1):22--55,
	2014.
	
	\bibitem{Peleg1973}
	Bezalel Peleg and Menahem~E. Yaari.
	\newblock On the existence of a consistent course of action when tastes are
	changing.
	\newblock {\em The Review of Economic Studies}, 40(3):391, July 1973.
	
	\bibitem{p90}
	Shige Peng.
	\newblock A general stochastic maximum principle for optimal control problems.
	\newblock {\em SIAM Journal on Control and Optimization}, 28(4):966--979, 1990.
	
	\bibitem{Phelps1968}
	E.~S. Phelps and R.~A. Pollak.
	\newblock On second-best national saving and game-equilibrium growth.
	\newblock {\em The Review of Economic Studies}, 35(2):185, April 1968.
	
	\bibitem{p18}
	Chi~Seng Pun.
	\newblock Robust time-inconsistent stochastic control problems.
	\newblock {\em Automatica}, 94:249--257, 2018.
	
	\bibitem{Strotz1955}
	Robert~H. Strotz.
	\newblock Myopia and inconsistency in dynamic utility maximization.
	\newblock {\em The Review of Economic Studies}, 23(3):165, December 1955.
	
	\bibitem{sun2016}
	Jingrui Sun, Xun Li, and Jiongmin Yong.
	\newblock Open-loop and closed-loop solvabilities for stochastic linear
	quadratic optimal control problems.
	\newblock {\em SIAM Journal on Control and Optimization}, 54(5):2274--2308,
	2016.
	
	\bibitem{ves03}
	WA~Van Den~Broek, JC~Engwerda, and Johannes~M Schumacher.
	\newblock Robust equilibria in indefinite linear-quadratic differential games.
	\newblock {\em Journal of Optimization Theory and Applications},
	119(3):565--595, 2003.
	
	\bibitem{w45}
	Abraham Wald.
	\newblock Statistical decision functions which minimize the maximum risk.
	\newblock {\em Annals of Mathematics}, pages 265--280, 1945.
	
	\bibitem{wang2020}
	Tianxiao Wang.
	\newblock Equilibrium controls in time inconsistent stochastic linear quadratic
	problems.
	\newblock {\em Applied Mathematics \& Optimization}, 81(2):591--619, 2020.
	
	\bibitem{Wei2017}
	Qingmeng Wei, Jiongmin Yong, and Zhiyong Yu.
	\newblock Time-inconsistent recursive stochastic optimal control problems.
	\newblock {\em {SIAM} Journal on Control and Optimization}, 55(6):4156--4201,
	January 2017.
	
	\bibitem{Yong2012}
	Jiongmin Yong.
	\newblock Time-inconsistent optimal control problems and the equilibrium {HJB}
	equation.
	\newblock {\em Mathematical Control {\&} Related Fields}, 2(3):271--329,
	September 2012.
	
	\bibitem{yz99}
	Jiongmin Yong and Xun~Yu Zhou.
	\newblock {\em Stochastic controls: Hamiltonian systems and {HJB} equations},
	volume~43.
	\newblock Springer Science \& Business Media, New York, NY, 1999.
	
\end{thebibliography}

\end{document}